\documentclass[hidelinks,11pt]{article}
\usepackage{amsmath,amssymb,amsfonts,amsthm,mathrsfs,relsize,mathtools,graphicx,float,hyperref,yhmath,stmaryrd,url,multicol,authblk,fullpage,bbm,euscript,enumitem,accsupp}
\usepackage[labelsep=period,font=footnotesize,labelfont=bf]{caption}
\usepackage[title]{appendix}
\usepackage{titlesec}
\titlelabel{\thetitle.\,\,}
\allowdisplaybreaks
\usepackage{pgf,tikz}
\usetikzlibrary{arrows}
\usetikzlibrary[patterns]
\setlist[itemize]{noitemsep,topsep=0pt}

\let\svthefootnote\thefootnote
\newcommand\blankfootnote[1]{%
\let\thefootnote\relax\footnotetext{#1}%
\let\thefootnote\svthefootnote%
}

\theoremstyle{plain}
\newtheorem{theorem}{Theorem}[section]
\newtheorem{lemma}[theorem]{Lemma}
\newtheorem{corollary}[theorem]{Corollary}
\newtheorem{observation}[theorem]{Observation}

\theoremstyle{definition}
\newtheorem{case}[theorem]{Case}

\newcommand{\da}{\rotatebox[origin=c]{45}{$\Box$}}

\DeclareMathAlphabet{\mathbbmsl}{U}{bbm}{m}{sl}

\DeclareMathAlphabet{\mathpzc}{OT1}{pzc}{m}{it}
\DeclareMathAlphabet{\mathsfit}{T1}{\sfdefault}{\mddefault}{\sldefault}\SetMathAlphabet{\mathsfit}{bold}{T1}{\sfdefault}{\bfdefault}{\sldefault}

\providecommand*{\napprox}{%
\BeginAccSupp{method=hex,unicode,ActualText=2249}%
\not\approx
\EndAccSupp{}%
}

\begin{document}

\title{\bf{The weak saturation number of  $\boldsymbol{K_{2, t}}$}\\ \vspace{9mm}}

\author{
Meysam Miralaei$^{^{1, a}}$ \, \,  Ali Mohammadian$^{^{2, b, c}}$ \, \,  Behruz Tayfeh-Rezaie$^{^{1, b}}$ \\
$^{^1}$School of Mathematics,    Institute for Research in
Fundamental  Sciences (IPM),  \\  P.O. Box 19395-5746, Tehran, Iran \\ \vspace{2mm}
$^{^2}$School of Mathematical Sciences,    Anhui University, \\
Hefei 230601,  Anhui,    China \\ \vspace{2mm}
\href{mailto:m.miralaei@ipm.ir}{m.miralaei@ipm.ir} \qquad  \href{mailto:ali\_m@ahu.edu.cn}{ali\_m@ahu.edu.cn} \qquad     \href{mailto:tayfeh-r@ipm.ir}{tayfeh-r@ipm.ir}\\ \vspace{7mm}}

\blankfootnote{\hspace*{-6mm}$^{^a}$Partially  supported by a grant from IPM.\\
$^{^b}$Partially  supported by Iran  National  Science Foundation   under project number  99003814.\\
$^{^c}$Partially  supported   by the    Natural Science Foundation of Anhui Province  with  grant identifier 2008085MA03 and by the National Natural Science Foundation of China with  grant number 12171002.}

\date{}

\maketitle

\begin{abstract}
For   two graphs $G$ and $F$, we say  that  $G$   is  weakly $F$-saturated
if $G$  has    no   copy of  $F$  as a   subgraph and one  can join all  the   nonadjacent pairs of vertices
of $G$ in some order so    that a new copy of $F$  is created at each step.
The  weak saturation number   $\mathrm{wsat}(n, F)$ is the
minimum number of edges of  a weakly $F$-saturated graph on $n$ vertices.
In this paper, we examine $\mathrm{wsat}(n, K_{s, t})$, where $K_{s, t}$ is the complete bipartite graph with parts of  sizes $s$ and $ t $.
We determine
$\mathrm{wsat}(n, K_{2, t})$ for all $n\geqslant t+2$ which particularly     corrects    a previous report in the literature.
It is  also shown    that $\mathrm{wsat}(s+t, K_{s,t})=\binom{s+t-1}{2}$ if $\gcd(s, t)=1$  and $\mathrm{wsat}(s+t, K_{s,t})=\binom{s+t-1}{2}+1$    otherwise.\\[-2mm]


\vspace{9mm}
\end{abstract}

\section{Introduction}

All graphs throughout this paper are finite, undirected, and without loops or multiple edges.  The
edge set of a graph $G$ is  denoted by  $E(G)$.
For   given two  graphs $G$ and $F$, a spanning subgraph $H$ of $G$ is said to be a  {\sl weakly $F$-saturated subgraph}  of $G$  if $H$ has no copy of  $F$   as a  subgraph and   there is an ordering $e_1, e_2, \dots$  of edges in $E(G)\setminus E(H)$ such that for $i=1, 2, \ldots$  the addition of $e_i$ to the spanning subgraph of $G$  with the edge set  $E(H)\cup\{e_1, \ldots, e_{i-1}\}$ creates a  copy    of   $F$ that contains   $e_i$.
The minimum number of edges in a weakly  $F$-saturated subgraph of $G$ is called the {\sl weak saturation number}  of $F$ in $G$ and  is denoted by $\mathrm{wsat}(G, F)$.
For the purpose of simplification, a weakly $F$-saturated subgraph of $K_n$ is said to be  a {\sl weakly $F$-saturated  graph} and    $\mathrm{wsat}(K_n, F)$ is written as   $\mathrm{wsat}(n, F)$, where   $K_n$ is the complete   graph on   $n$ vertices.
For example, each  path  graph
is    weakly $K_3$-saturated and it is easily seen  that $\mathrm{wsat}(n, K_3)=n-1$ due to the connectivity.

Determining the exact value of  $\mathrm{wsat}(n, F)$   for a given graph $F$ is  often  quite difficult.
It is worth mentioning that the study of any  extremal parameter    is an important task in    graph theory and often receives  a great deal of  attention.
Weak saturation is closely related
to the so-called  `graph bootstrap percolation'    which  was    introduced for the first time in  \cite{Balogh}.
The notion of  weak saturation  was  initially introduced by Bollob\'{a}s \cite{Bella} in 1968.
Although  the  weak  saturation number has   been   studied for a long time, related   literature is still poor. Indeed, the main difficulty lies in proving lower bounds where usually combinatorial methods do not seem to work. Most    arguments that have been used in this area    are   based on algebraic methods.
However, our proofs in the current paper are all combinatorial.
For    results on  weak saturation and related topics, we refer to the survey \cite{Faud.3}.

Lov\'{a}sz \cite{Lovasz} proved   that      $\mathrm{wsat}(n, K_r)=(r-2)n-\binom{r-1}{2}$    when    $n\geqslant r\geqslant2$,    settling a conjecture of Bollob\'{a}s  \cite{Bella}. The   result is   also proved by   Frankl \cite{Frankl}, Kalai \cite{Kalai1}, Alon \cite{Alon},   and Yu \cite{Yu}.
Surprisingly, these proofs all  are   based on algebraic techniques and    no combinatorial  proof has been found so far.

After complete graphs, the next most natural problem to consider regarding weak saturation numbers is    description  of  the   behavior    of  $\mathrm{wsat}(n, K_{s,t})$, where $K_{s,t}$ is the complete bipartite graph with parts of  sizes $s$ and $ t $. Borowiecki and  Sidorowicz \cite{Borowiecki} proved that  $\mathrm{wsat}(n, K_{1,t})=\binom{t}{2}$ provided    $n\geqslant t+1$.   A short proof of this  result is given in  \cite{Faudree}.
The equality  $\mathrm{wsat}(n, K_{2,2})=n$ follows  from Theorem 16 of \cite{Borowiecki} for all   $n\geqslant4$.
Faudree,  Gould, and  Jacobson   \cite{Faudree} showed that $\mathrm{wsat}(n, K_{2,3})=n+1$ for   all $n\geqslant5$. Using multilinear algebra,
Kalai \cite{Kalai2}   established  that   $\mathrm{wsat}(n, K_{t, t})=(t-1)n-\binom{t-1}{2}$  if   $n\geqslant4t-4$.
This result  was also proved by Kronenberg, Martins,  and Morrison   \cite{Kronenberg} for every  $n\geqslant3t-3$ by a linear algebraic argument. They also determined $\mathrm{wsat}(n, K_{t,t+1})$ for any $ n\geqslant 3t-3$.

The authors of \cite{cui} claimed that they determined $\mathrm{wsat}(n, K_{2, t})$     for   $t\geqslant4$ and  $n\geqslant2t-1$,
but missing and confusing proof details prevent complete confidence in the result.
In particular, ambiguous wording in  their half page argument to prove the lower bound on $\mathrm{wsat}(n, K_{2, t}) $ makes the proof difficult to follow.
We believe that this problem is not already solved.
In the current    paper, we fill  this  gap  in the literature  by  proving   the following result.

\begin{theorem}\label{main.thm}
For every two integers $n, t$ with  $t\geqslant3$ and $n\geqslant t+2$, the following statements  hold.
\begin{itemize}
\item[{\rm (i)}] If $ t $ is odd, then  $\mathrm{wsat}(n, K_{2, t})=n-2+\binom{t}{2}$.
\item[{\rm (ii)}] If $ t $ is even and $n\leqslant2t-2$, then  $\mathrm{wsat}(n, K_{2, t})= n-1+\binom{t}{2}$.
\item[{\rm (iii)}] If $ t $ is even and $n\geqslant2t-1$, then  $\mathrm{wsat}(n, K_{2, t})=n-2+\binom{t}{2}$.
\end{itemize}
\end{theorem}

The proofs  of the lower bounds  of  Theorem \ref{main.thm}  which are  presented
in Section \ref{aval123}  form   the most involved part   of the paper. In Section \ref{span123}, we establish   the following    theorem which particularly proves   Theorem \ref{main.thm} for  the  initial   case    $n=t+2$.
Generally,     determination of  $\mathrm{wsat}(n, F)$ for  graphs   $F$ on   $n$ vertices seems to be an attractive problem.

\begin{theorem}\label{Kst}
For every two    positive  integers $s$ and $t$,
\begin{eqnarray*}
\mathrm{wsat}(s+t, K_{s, t})=\left\{
\begin{array}{ll}\vspace{-4mm}&\\
\mathlarger{\binom{s+t-1}{2}}  &  \quad  \text{\large if }   \mathlarger{\gcd(s,t)=1}\text{\large ,} \\ \\
\mathlarger{\binom{s+t-1}{2}+1}  & \quad   \text{\large otherwise.}\\\vspace{-3.75mm}&
\end{array}\right.
\end{eqnarray*}
\end{theorem}

A relatively new trend in extremal graph theory is to extend the classical deterministic results  to random analogues. Such  study reveals the behavior of extremal parameters for a typical graph. For  instance,  the problem  of  determination of    $\mathrm{wsat}(\mathbbmsl{G}(n, p), K_{s,t})$ for given fixed integers $s$ and $ t$    is  still   unsolved   in general case, where  $\mathbbmsl{G}(n, p)$ denotes    the Erd\H{o}s--R\'{e}nyi  random graph model. Kalinichenko and
Zhukovskii  \cite{Kalinichenko}  presented   some  sufficient  conditions  for which   $\mathrm{wsat}(\mathbbmsl{G}(n, p), F)=\mathrm{wsat}(n, F)$ with high probability.
Theorem \ref{main.thm}  combined with Corollary 1 of \cite{Kalinichenko} yields   that with high probability $\mathrm{wsat}(\mathbbmsl{G}(n, p),K_{2,t})=n-2+\binom{t}{2}$ for each   constant  $p\in(0, 1)$.

Below, we introduce more  notations  and terminologies       that we use in the rest of the  paper.
Let $G$ be a graph. The vertex set of $G$ is  denoted by $V(G)$  and the {\sl order}  of $G$ is defined  as $|V(G)|$. We  set  $e(G)=|E(G)|$.
For every two adjacent  vertices $u$ and $v$,  we denote the edge joining    $u$ and $v$ by $uv$.
The {\sl complement}  of $G$, denoted by $\overline{G}$,   is a graph with vertex set $V(G)$
in which $uv\in E(\overline{G})$     if $u\neq v$ and  $uv\notin E(G)$.
For a  subset   $X$ of $V(G)$, we denote the induced subgraph of $G$ on $X$  by $G[X]$.
For a  subset   $Y$ of $E(G)$, we denote by $G-Y$  the   graph obtained from  $G$ by removing the edges in  $Y$.
For a  subset   $Z$ of $E(\overline{G})$,   we adopt the notation  $G+Z$  to denote  the graph with vertex set $V(G)$ and edge set $E(G)\cup Z$.
For simplicity, we write $G-e$ instead of $G-\{e\}$ and  $G+e$ instead  of $G+\{e\}$.
For a vertex     $v$ of $G$, denote by $G-v$ the graph obtained from $G$   by removing  $v$  and all edges incident to $v$.
Also, define the set of neighbors of $v$ as       $N_G(v)=\{x\in V(G) \, | \, x \text{   is adjacent to } v\}$ and   the   {\sl degree}  of $v$    as $\deg_G(v)=|N_G(v)|$.
The   minimum degree of vertices of $ G$ is  denoted by  $ \delta(G)$.
For the sake of convenience, we set $N_G[u]=\{u\}\cup N_G(u)$ and $N_G(u, v)=N_G(u)\cap N_G(v)$.
For every two subsets $A$ and $B$ of $V(G)$, let $E_G(A, B)$ denote the set  of  edges of $G$  having  an   endpoint  in $A$ and the other endpoint  in $B$. We set $e_G(A, B)=|E_G(A, B)|$.
For simplicity, we write $E_G(v, A)$ instead of $E_G(\{v\}, A)$,
$E_G(A)$ instead of $E_G(A, A)$ and $e_G(A)$ instead of $e_G(A, A)$.
The {\sl    union} of  two vertex disjoint graphs   $G_1$ and $G_2$,   denoted by
$G_1\sqcup  G_2$,   is the graph  with  the  vertex   set  $V(G_1)\cup V(G_2)$  and the  edge  set    $E(G_1)\cup E(G_2)$.
The {\sl  join}  of two vertex disjoint graphs $G_1$ and $G_2$,   denoted by  $G_1\vee  G_2$,  is the graph  obtained from $G_1\sqcup G_2$ by
joining every  vertex in $V(G_1)$ to every vertex in $V(G_2)$.

\section{Determination of $\boldsymbol{\mathrm{wsat}(s+t, K_{s,t})}$}\label{span123}

The problem  of determining  $\mathrm{wsat}(n, F)$ when the graph   $F$ is of order $n$  is interesting to explore.
In this section, we solve this problem when $F$ is a  complete bipartite graph.
The following lemma helps us  to get a   lower bound.

\begin{lemma}\label{complement}
Let $s, t$ be       positive integers and let $G$ be a weakly $K_{s,t}$-saturated graph of order  $s+t$. Then, $\overline{G}$ has no cycle. Moreover,   if $\gcd(s, t)\neq 1$, then $\overline{G}$ is disconnected.
\end{lemma}

\begin{proof}
Fix an order $ e_1, e_2, \ldots $ of  $ E(\overline{G}) $ that is obtained from a weakly $ K_{s,t} $-saturation process on $ G $.
By contradiction, suppose that $ \overline{G} $ has a cycle, say $ C $. Let $ e_i $ be the first edge of $ C $ that  appears in the order $ e_1, e_2, \ldots$ In view of    the definition of weakly $ K_{s,t} $-saturation process,  there is  a partition $ \{A, B\} $ of $ V(G) $ with $ |A|=s $ and $ |B|=t $ such that $ e_i $ is the only missing edge between $ A $ and $ B $ in $G+\{e_1, \ldots, e_{i-1}\}$. So, both endpoints of each  edge among  $e_{i+1}, e_{i+2}, \ldots$ belong to one of   $ A $ and   $ B $. This is impossible,  since $ C $ has to  pass  through at least one   edge $ e_j$ with  $ j>i $ having endpoints in both   $ A $ and   $ B $. This  shows   that  $\overline{G}$ has no cycle.

Now, assume that $\overline{G} $ is connected. As we saw above,  $ \overline{G}$ is a tree.  Let $H_0=\overline{G}$ and    $ H_i=\overline{G}-\{e_1, \ldots, e_i\}$  for any $i\geqslant1$.
We claim that for any $ i\geqslant0 $,  $ H_i$ is a forest whose connected  components are of order divisible by $ d $, where    $d=\gcd(s,t)$.  Since   $\overline{G}-\{e_1, e_2, \ldots\} =\overline{K_{s+t}}$,  we find that  $d=1 $, as required.

We prove the claim by induction on $ i $. The claim is clearly valid   for $ i=0 $.  So,  assume that $ i\geqslant 1 $.
According to the  definition of weakly $ K_{s,t} $-saturation process, there is a partition $ \{A, B\} $ of $ V(G) $ with $ |A|=s $ and $ |B|=t $ such that $ e_i $ is the only missing edge between $ A $ and $ B $ in  $G+\{e_1, \ldots, e_{i-1}\}$. Hence,  $ e_i $ is the only edge  in $ H_{i-1} $  between $ A $ and $ B $.
Let $ C_1, \ldots, C_i $ be the connected  components of $ H_{i-1} $. Without loss of generality,  assume that $ e_i\in E(C_1) $. So,  the connected  components of $ H_i $ are $ C_1', C_1'', C_2, \ldots, C_i $, where $ C_1' $ and $ C_1'' $ are     respectively the induced subgraphs of $ C_1 $ on $A\cap  V(C_1)$ and $B\cap  V(C_1)$.
As   $ e_i $ is the only edge  in $ H_{i-1} $ between $ A $ and $ B $, we conclude that  either $ V(C_i)\subseteq A$ or $ V(C_i)\subseteq B $ for any   $i\geqslant2 $.
It follows that $ A $ is a disjoint union of  $ V(C'_1) $ and some sets among $ V(C_2), \ldots, V(C_i) $.  The induction hypothesis yields that  $ |V(C_2)|, \ldots, |V(C_i)| $ are multiples  of  $d$. This and the divisibility of  $|A|$    by $d$  imply    that $ | V(C'_1)| $ is a multiple of  $d$. A similar argument works  for  $ |V(C_1'')| $.  The claim is established.
\end{proof}

The following consequence immediately follows from  Lemma \ref{complement}.

\begin{corollary}\label{lower-I-II}
For every integers $s$ and $t$,   $\mathrm{wsat}(s+t, K_{s, t})\geqslant\binom{s+t-1}{2}$. Moreover, if $\gcd(s, t)\neq 1$, then $\mathrm{wsat}(s+t, K_{s, t})\geqslant\binom{s+t-1}{2}+1$.
\end{corollary}

We present  the following two lemmas  to obtain   a tight    upper bound.
We use   the notation  $P_n$ for   the  path   graph   of order  $n$.

\begin{lemma}\label{up-general}
Let $s$ and $t$ be     positive integers. Then,  $\mathrm{wsat}(s+t, K_{s,t})\leqslant \binom{s+t-1}{2}+1$.
\end{lemma}

\begin{proof}
We prove  that $ G= \overline{ P_{s+t-1}\sqcup K_1}  $ is weakly $ K_{s,t} $-saturated.	
Denote by  $v_1, \ldots, v_{s+t-1}$    the
vertices of $P_{s+t-1}$ going in the natural order of the path
and set $V(K_1)=\{v_{s+t} \}$. Let $ e_1=v_{s}v_{s+1} $, $ e_i=v_{i-1}v_i $ for $i=2, \ldots,   s $, and
$ e_i=v_iv_{i+1} $ for $i= s+1,  \ldots,  s+t-2 $. We claim  that  $   e_1, \ldots, e_{s+t-2} $ is an  order in which the weakly $ K_{s,t} $-saturation
process occurs.
Let $H_0=G$ and       $ H_i=G+\{e_1, \ldots, e_i\} $ for $i= 1, \ldots,  s+t-2 $. In order to prove the assertion, we find a
partition $\{A_i, B_i\}$ of $ V(G) $ such that   $|A_i|=s$,  $|B_i|=t$, and  $ e_{i+1} $ is the only missing edge between $ A_i $ and $ B_i $ in $ H_i $  for $i=0, 1, \ldots,  s+t-3 $. To do this, it is enough to   introduce $A_0, A_1, \ldots, A_{s+t-3}$.
Let
$A=\{v_1, \ldots, v_{s}\} $. Now, set    $ A_0=A $,
$A_i=(A\setminus \{v_i\})\cup\{v_{s+t}\}$  for $i=1, \ldots,   s-1 $, and
$A_i=(A\setminus \{v_1\})\cup\{v_{i+1}\}$  for $i= s,  \ldots,  s+t-3 $.
\end{proof}

\begin{lemma}\label{up-gcd=1}
Let $s$ and $t$ be   positive integers with  $\gcd(s,t)=1$. Then,  $\mathrm{wsat}(s+t, K_{s,t})\leqslant\binom{s+t-1}{2}$.
\end{lemma}

\begin{proof}
We prove  that $ G=\overline{P_{s+t}}$ is weakly $ K_{s,t} $-saturated.
We proceed by  induction on $ s+t $. The assertion clearly holds for $ s+t=2 $. Let $ s+t\geqslant 3 $  and denote by  $v_1, \ldots, v_{s+t}$    the
vertices of $P_{s+t}$ going in the natural order of the path.
Partition $ V(G) $ into two subsets  $ A=\{v_1, \ldots, v_{s}\} $ and $ B=\{v_{s+1}, \ldots, v_{s+t}\} $.
Since the edge $ e=v_sv_{s+1} $ is
the only missing edge between $ A $ and $ B $ in $ G $, we   may consider $e$   as
the first element in an    ordering  of  $E(\overline{G})$ in a   weakly $ K_{s,t} $-saturation process on $G$.  For the sake of convenience, let $G'=G+e$  and without loss of generality, assume that $t\geqslant s$.
Using the definition,   in each step of a weakly $ K_{s, t-s} $-saturation  process on $ G'[B] $, there is   a partition $\{C,  D\}$ of $B$  such that  $ |C|=s $, $ |D|=t-s $,  and all edges between $ C $ and $ D $ are present    except   exactly one.
Since there is no edge between $A$ and $B$ in $\overline{G'}$,
every   step of a weakly $ K_{s, t-s} $-saturation process on $ G'[B] $ corresponding to a vertex  partition $\{C,  D\}$
can be considered as a  step of a   weakly $ K_{s,t} $-saturation process  on $G'$   corresponding to the vertex   partition    $ \{C, A\cup D \}$.
Hence, by the induction hypothesis,   $ G'[B] $  may be completed to reach to   $ K_t $ through a  weakly $ K_{s, t} $-saturation process.
Thus, it remains  to show that $ G''=G'+\{v_iv_{i+1} \, | \,  s+1\leqslant i \leqslant s+t-1\} $ is weakly $ K_{s,t} $-saturated.
For   $  i=1, \ldots,  s-1 $,  the edge $ e_i=v_iv_{i+1} $ is the only missing edge between
$\{v_1, \ldots, v_i\} \cup\{ v_{s+1}, \ldots, v_{2s-i}\} $ and $ \{v_{i+1}, \ldots, v_{s}\} \cup\{v_{2s-i+1}, \ldots, v_{s+t}\} $ in $ G'' $ and therefore   we may add $e_i$ to $G''$ in the   weakly $ K_{s,t} $-saturation process.
\end{proof}

We end this section  by pointing out that    Theorem \ref{Kst} is immediately   concluded from Corollary \ref{lower-I-II},  Lemma \ref{up-general},  and  Lemma \ref{up-gcd=1}.

\section{Determination of $\boldsymbol{\mathrm{wsat}(n, K_{2,t})}$}\label{aval123}

In this section, we  establish   Theorem \ref{main.thm} which is a direct consequence  of   Lemmas \ref{lowerr}, \ref{generallower},  and \ref{event}.
The following lemma is  known, although it seems that it  is not  explicitly stated anywhere.
We include a proof here for the sake of completeness.

\begin{lemma}\label{pendent}
Let $F$ be a graph with  $\delta(F)\geqslant1$    and let $G$ be a weakly $F$-saturated graph such that $|V(G)|\geqslant |V(F)|-1$.  		Join  a new vertex $v$    to $\delta(F)-1$ arbitrary vertices of $G$. Then,  the resulting graph   is  also   weakly $F$-saturated.
\end{lemma}

\begin{proof}
Denote the   resulting graph  by $G'$.
Since  $G$ is  weakly $F$-saturated, we   may  add
all edges in $ \{uv \in E(\overline{G'}) \, | \, u,v\in V(G)\}$  to  $G'$ in  some order  to obtain a complete subgraph of $ G'$ on $V(G)$. Let $e\in E(F)$ be incident to a vertex of degree  $\delta(F)$.  For each vertex $x\in V(G)\setminus N_{G'}(v)$,   there is a copy of  $F-e$ in $G'$ containing the vertices in  $\{v, x\}\cup N_{G'}(v)$ and so,  we may  connect   $v$ to $x$ in the weakly $F$-saturation process on $G'$. The assertion follows.
\end{proof}

The following lemma proves the upper  bounds of Theorem \ref{main.thm}.

\begin{lemma}\label{lowerr}
For every two integers $n, t$ with  $t\geqslant3$ and $n\geqslant t+2$, the following statements  hold.
\begin{itemize}
\item[{\rm (i)}] If $t$ is odd, then  $\mathrm{wsat}(n, K_{2, t})\leqslant n-2+\binom{t}{2}$.
\item[{\rm (ii)}] If $t$ is even and $n\leqslant 2t-2$, then  $\mathrm{wsat}(n, K_{2, t})\leqslant n-1+\binom{t}{2}$.
\item[{\rm (iii)}] If $t$ is even and $n\geqslant 2t-1$, then  $\mathrm{wsat}(n, K_{2, t})\leqslant n-2+\binom{t}{2}$.
\end{itemize}
\end{lemma}

\begin{proof}
Let $H$ be a weakly $K_{2,t}$-saturated graph of order   $t+2$  with  $\mathrm{wsat}(t+2, K_{2,t})$ edges. Attach  $n-t-2$ pendent vertices to  an  arbitrary vertex   of $H$ to obtain   a  graph $G$ of order $n$. By Lemma \ref{pendent}, $G$ is a weakly $K_{2,t}$-saturated graph with $n-t-2+\mathrm{wsat}(t+2, K_{2,t})$ edges. Parts  (i) and  (ii) follow from Theorem \ref{Kst}.
The graph  $\mathbbmsl{G}_{n, t}$,  depicted in Figure \ref{figM1} and  introduced in \cite{cui, Kronenberg},  is   weakly $K_{2,t}$-saturated  for    $n\geqslant 2t-1$.    This can be  proved  by using a proof similar to that of    Proposition 14 in \cite{Kronenberg}.
Since   $\mathbbmsl{G}_{n, t}$  has $n$ vertices and  $n-2+\binom{t}{2}$ edges,   (iii) follows.
\end{proof}

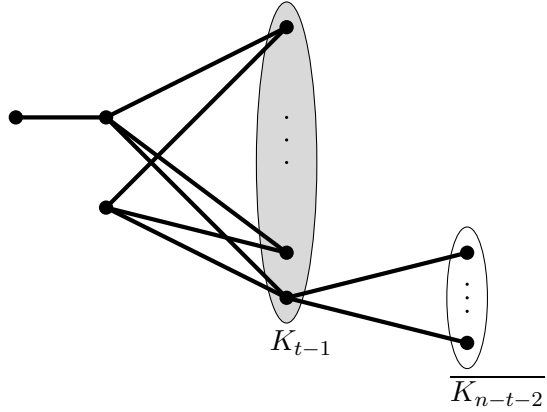
\begin{figure}[H]
\begin{center}
\begin{tikzpicture}[line cap=round,line join=round,>=triangle 45,x=1.0cm,y=1.0cm,scale=0.6]
\draw [rotate around={90:(2,-8)},line width=0.4pt,color=black,fill=gray!30] (2,-8) ellipse (3.56cm and 0.67cm);
\draw [rotate around={90:(6,-11)},line width=0.4pt,color=black] (6,-11) ellipse (1.57cm and 0.45cm);
\draw (1.4,-11.5) node[anchor=north west]{$K_{t-1}$};
\draw (5.4,-12.5) node[anchor=north west]{$\overline{K_{n-t-2}}$};
\clip(-4.61,-13.67) rectangle (7.58,-4.33);
\draw [line width=1.6pt] (-4,-7)-- (-2,-7);
\draw [line width=1.6pt] (-2,-7)-- (2,-5);
\draw [line width=1.6pt] (-2,-9)-- (2,-5);
\draw [line width=1.6pt] (-2,-9)-- (2,-10);
\draw [line width=1.6pt] (-2,-9)-- (2,-11);
\draw [line width=1.6pt] (-2,-7)-- (2,-10);
\draw [line width=1.6pt] (-2,-7)-- (2,-11);
\draw [line width=1.6pt] (2,-11)-- (6,-10);
\draw [line width=1.6pt] (2,-11)-- (6,-12);
\begin{scriptsize}
\fill [color=black] (-2,-7) circle (4.5pt);
\fill [color=black] (-4,-7) circle (4.5pt);
\fill [color=black] (-2,-9) circle (4.5pt);
\fill [color=black] (2,-5) circle (4.5pt);
\fill [color=black] (2,-11) circle (4.5pt);
\fill [color=black] (2,-10) circle (4.5pt);
\fill [color=black] (6,-10) circle (4.5pt);
\fill [color=black] (6,-12) circle (4.5pt);
\fill [color=black] (2,-7) circle (1.0pt);
\fill [color=black] (2,-8) circle (1.0pt);
\fill [color=black] (2,-7.5) circle (1.0pt);
\fill [color=black] (6,-10.7) circle (1.0pt);
\fill [color=black] (6,-11) circle (1.0pt);
\fill [color=black] (6,-11.3) circle (1.0pt);
\end{scriptsize}
\end{tikzpicture}
\caption{The graph  $\mathbbmsl{G}_{n, t}$. We have not drawn    the edges between  the   vertices in the  gray elliptical disk.}\label{figM1}
\end{center}
\end{figure}

The following observation is trivially true. We state it   for clarity.

\begin{observation}\label{substit}
Fix a graph $F$ and let   $G, H$ be two graphs with  the same vertex set. Assume that  the graph   obtained from $H$ by adding   a sequence of edges in a weakly $F$-saturation process contains $G$ as a subgraph. If $G$  is weakly $F$-saturated, then so is $H$.
\end{observation}

The following lemma establishes  a general  lower     bound on    $\mathrm{wsat}(n, K_{2, t})$.

\begin{lemma}\label{generallower}
Let $t\geqslant3$ and $n\geqslant t+2$. Then,  $\mathrm{wsat}(n, K_{2, t})\geqslant n-2+\binom{t}{2}$.
\end{lemma}

\begin{proof}
Let $ G_0 $ be a weakly $ K_{2,t} $-saturated graph. So,   $ G_0 $ is connected.
By Lemma \ref{pendent}, if we attach a new pendent vertex to a vertex of $ G_0 $, then the resulting graph is also weakly $ K_{2,t}$-saturated
whose number of vertices  is one more  than the number of vertices of $G_0$  and whose number of edges is one more  than the number of edges of $G_0$.
We attach $ t^3$ new pendent vertices to each vertex of $ G_0 $ and we call the resulting graph by $ G$. To prove the assertion, it suffices to show that
$e(G)\geqslant n-2+\binom{t}{2}$ provided  $ n=|V(G)|$.

We define a process in which step, $ G $ is updated so that a special structure on $ G $  is preserved.     In each step of the  process, $ G $ looks as follows. The graph $ G$ contains $ G_0$ as a subgraph.
The edges of $ G $ are colored  by two colors    black and red. At the beginning of the    process,  all edges are black  and the number of black edges does not change during the process.  The  spanning subgraph of $ G $ induced on   black edges is   connected.
There exist     two disjoint subsets $ A$  and   $B $ of $V(G)$ which  are  equipped with  the following features.
There is an ordering  on  $ A $ under which   the vertices  in  $A$ can be arranged  as   $ x_1, \ldots, x_k $, where $ k=|A| $.
Every red edge has at least one  endpoint in $ A $.
There exist the     partition $\{A_{1}, \ldots, A_{m}\}$ of $A$ and  the  partition  $\{B_{1}, \ldots, B_{m}\}$ of $B$ which   are  described below.  Let $i\in\{1, \ldots, m\}$ and denote by $ x_{s_{i}} $ the first element among  $ x_1, \ldots, x_k $ which  appears within $ A_{i}$.
The following properties will be held  in each step of the  process.
\begin{itemize}[label=*****, wide=0pt, leftmargin=*]
\item[(P.1)] For every two  vertices  $ x, y\in A_i $, we have $ N_{G}(x)\setminus\{y\}=N_{G}(y)\setminus \{x\} $.
\item[(P.2)] The set $A_i$ is either a clique of size at least $2$  or an independent set of size $2$ or  $3$.
\item[(P.3)] Every edge between $x_{s_i}$ and $A_i\setminus\{x_{s_i}\}$ is black   whenever  $A_i$ is   a clique.
\item[(P.4)]  Every   vertex in $ A_i $ is adjacent to every vertex in $ A_j $  whenever $i\neq j$.
\item[(P.5)]  Each  edge between    $ x_{s_i} $  and $V(G)\setminus A$ is black.
\item[(P.6)] Every    vertex $x_r \in A_i$ is adjacent to every  vertex   $ x_{s_j} $  by a black edge   if   $i\neq j$ and $s_j<r$.
\item[(P.7)]  The size of   $ B_{i}$ is $t-s_{i}$.
\item[(P.8)]  For any  vertex $ x\in B_i $, we have $ N_{G}(x)=A_i $.
\item[(P.9)]  Any vertex $ x_r\in A_{i} $ is adjacent to exactly $ \alpha_r $ vertices in $ B_{i} $ by black edges, where
\begin{eqnarray*}
\alpha_r=\left\{
\begin{array}{ll}
t-r &  \quad  \, \, \,  \text{if }   r=s_i,\\ \vspace{-3mm} \\
t-r+2 & \quad   \begin{array}{l}\text{if $A_{i}$ is an independent set of size} \\ \text{3 and  $x_r$  is the third element of  $A_i$,}\end{array}\\ \vspace{-3mm} \\
t-r+1 & \quad  \, \, \,   \text{otherwise}.
\end{array}\right.
\end{eqnarray*}
\item[(P.10)] We have $s_1<\cdots<s_m$.
\end{itemize}
The configuration described  above is designed so that at each step  of the process, the graph induced on    black edges is  weakly $K_{2,t}$-saturated,
and      red edges  are  the ones that are added    through     the  weakly $K_{2,t}$-saturation process.
This will be  shown in Cases \ref{Case I}, \ref{Case II}, and \ref{Case III} below. Actually, the following statement preserved during the process:
For every two vertices $x,y\in V(G)$ with $|N_G(x,y)|\geqslant t-1$, there are at least $t-1$ vertices in $N_G(x,y)$ are all adjacent to both $x$ and $y$ by black edges
as well as any  vertex in $N_G(x,y)$ is adjacent to at least one of  $x$ and  $y$ by a black edge.

At the beginning of the    process,  all edges are black and    $A=B=\varnothing$. In  Cases \ref{Case I}, \ref{Case II}, and \ref{Case III},   we   explain       how in each step of the  process we update $ G$ and $ A, B, C$ to proceed to the next step, where $ C=V(G)\setminus (A\cup B)$.
In each step of the process, at most two vertices will be added to $A$.
So, we may repeat the  process until $k\leqslant t+1$. More precisely, the process will be terminated whenever one of the following occurs.
\begin{itemize}[label=*****, wide=0pt, leftmargin=*]
\item[(T.1)] $k\leqslant t+1$ and $ |A_i|\geqslant t-1 $ for some $ i\in\{1, \ldots, m\} $.
\item[(T.2)] $ k=t $, $ |A_i|\leqslant t-2 $ for  $ i=1, \ldots, m$,  and there are two vertices $ u,v\in C $ such that $|N_{G}(u, v)|\geqslant t-1 $ and $ N_{G}(u)\setminus\{v\}\neq N_{G}(v)\setminus\{u\}$.
\item[(T.3)] $ k=t+1 $, $ |A_i|\leqslant t-2 $  for  $ i=1, \ldots, m$,  and  there  are two vertices  $ u, v \in C $ such that $|N_{G}(u, v)|\geqslant t-1 $ and $ N_{G}(u)\setminus\{v\}\neq N_{G}(v)\setminus\{u\}$.
\item[(T.4)] $k=t+1 $, $|A_i|\leqslant t-2 $ for  $ i=1, \ldots, m$,  and there are two vertices $ u \in A$ and $v\in C$ such that  $|N_{G}(u, v)|\geqslant t-1 $ and $ N_{G}(u)\setminus\{v\}\neq N_{G}(v)\setminus\{u\}$.
\end{itemize}

We now show what we do in each   step of  the  process before termination. At the beginning of each  step, $ G $ is weakly $ K_{2,t} $-saturated and so there are two vertices $ a, b $
such that  $ |N_{G}(a, b)|\geqslant t-1 $  and $ N_{G}(a)\setminus\{b\}\neq N_{G}(b)\setminus\{a\} $. As  (T.1)  is not   happened,  (P.8) forces that   $a,  b\notin B$.

\begin{case}\label{Case I}
$a, b\in C $.
\end{case}

\begin{proof}[Description]
As  (T.1)--(T.3) are   not   happened,    $ k\leqslant t-1 $. Set $ x_{k+1}=a $, $ x_{k+2}=b $,    $ s_{m+1}=k+1$,  and $ A_{m+1}=\{x_{k+1}, x_{k+2}\} $.
Suppose  that   $ x_{k+1}$ or $ x_{k+2} $ is not adjacent to $ x_{s_i} $ for some $ i\in \{1, \ldots, m\} $. We find from (P.1) that
$ N_{G}(x_{k+1}, x_{k+2})\cap A_i =\varnothing$. Therefore,
$|N_{G}(x_{k+1}, x_{k+2})\setminus A|\geqslant t-1-|N_{G}(x_{k+1}, x_{k+2})\cap A| \geqslant  t-1 -(k-|A_i|)\geqslant |A_i| \geqslant 2$.
So, we may   remove two arbitrary edges between   $ x_{k+2} $  and   $ N_{G}(x_{k+1}, x_{k+2})\setminus A $ and join both  $ x_{k+1} $ and $ x_{k+2} $ to $ x_{s_i} $ by black
edges and to all vertices in $ A_i\setminus \{x_{s_i}\}$ by red edges.
By repeating this, we derive  that   $A\subseteq N_{G}(x_{k+1}, x_{k+2})$ and
$|N_{G}(x_{k+1}, x_{k+2})\setminus A|\geqslant t-1-k$. Thus, (P.4) and (P.6) hold yet if $A_{m+1}$ is added to  $\{A_1, \ldots, A_m\}$.
We    remove $ t-1-k$ arbitrary  edges   between    $ x_{k+2} $  and   $ N_{G}(x_{k+1}, x_{k+2})\setminus A $ and
connect $ x_{k+2} $ to  all vertices in a    subset  $B_{m+1}$ consisting  of   $ t-1-k$   arbitrary    pendent vertices   in $ N_{G}(x_{k+1})$. Hence, (P.7)--(P.9) hold yet if $B_{m+1}$ is added to  $\{B_1, \ldots, B_m\}$.

Now, update $ A $ to $ A\cup A_{m+1} $ with the partition $ \{A_1, \ldots, A_{m+1}\} $ and update $ B $ to $ B\cup B_{m+1}$ with the partition $ \{B_1, \ldots, B_{m+1}\} $.
\renewcommand{\qedsymbol}{$\da$}
\end{proof}

\begin{case}\label{Case II}
$a\in A_i $ for some $ i \in  \{1, \ldots,  m\}$ and $ b\in C$.
\end{case}

\begin{proof}[Description]
Since (T.1) and (T.4)  are    not   happened,   $ k\leqslant t $.
In view of  (P.1) and  without loss of generality,  we may assume that $ a=x_{s_i} $.
Let $ x_{k+1}=b $.

First, assume that $ A_i\cup \{x_{k+1}\} $ is an independent set of size $ 3 $.
Suppose that  $ x_{k+1} $ is not adjacent to $ x_{s_j} $ for some $ j\in \{1, \ldots,  m\}\setminus \{i\} $.
We obtain from (P.1) that
$ N_{G}(x_{s_i}, x_{k+1})\cap A_j =\varnothing$. Therefore,
$|N_{G}(x_{s_i}, x_{k+1})\setminus A|\geqslant t-1-|N_{G}(x_{s_i}, x_{k+1})\cap A|
\geqslant t-1 -(k-|A_j|)\geqslant |A_j|-1\geqslant 1$.
So, we may  remove an  arbitrary  edge  between  $ x_{k+1} $   and   $ N_{G}(x_{s_i}, x_{k+1})\setminus A $ and join $ x_{k+1} $ to $ x_{s_j} $ by a black edge and to  all vertices in  $ A_j\setminus \{x_{s_j}\} $ by red edges. By repeating this, we  derive that
$A\setminus A_i \subseteq N_{G}(x_{s_i}, x_{k+1})$ and
$ |N_{G}(x_{s_i}, x_{k+1})\setminus A|\geqslant t-1 -(k-|A_i|)\geqslant  t-k+1 $.
Thus, (P.4) and (P.6) hold yet if  $x_{k+1}$ is added to $A_i$.
We now  remove $ t-k+1 $ arbitrary edges   between    $ x_{k+1} $     and   $ N_{G}(x_{s_i}, x_{k+1})\setminus A $ and
connect $ x_{k+1} $ to $ t-k+1 $ arbitrary distinct vertices in $ B_i $ by black edges. This is possible, since $ |B_i|\geqslant t-k+1$ by (P.7) and using  $s_i\leqslant k-1 $.
Hence, (P.8) and (P.9) hold  yet if  $x_{k+1}$ is added to $A_i$.

Next, assume that
$A_i\cup \{x_{k+1}\} $ is not an independent set of size $ 3 $.
Let $ x_{k+1} $ be  not adjacent to $ x_{s_j} $ for some $ j\in\{1, \ldots, m\} $.
We find from (P.1) that
$ N_{G}(x_{s_i}, x_{k+1})\cap A_j =\varnothing$. Therefore,
$|N_{G}(x_{s_i}, x_{k+1})\setminus A|\geqslant t-1-|N_{G}(x_{s_i}, x_{k+1})\cap A|
\geqslant t-1 -(k-|A_j|)\geqslant |A_j|-1\geqslant 1$. So,
we may   remove an arbitrary  edge    between   $ x_{k+1} $     and   $ N_{G}(x_{s_i}, x_{k+1})\setminus A $ and join $ x_{k+1} $ to $ x_{s_j} $ by a black edge and to all vertices in  $ A_j\setminus \{x_{s_j}\} $ by red edges.
By repeating this, we  derive  that   $ A \subseteq N_{G}(x_{k+1})$.
Hence, (P.4) and (P.6) hold yet if  $x_{k+1}$ is added to $A_i$.
If $ A_i $ is a clique, then $ A\setminus \{x_{s_i}\} \subseteq N_{G}(x_{s_i}, x_{k+1}) $.
Suppose that  $ A_i $ is  an independent set. It follows from $ A \subseteq N_{G}(x_{k+1})$ that
$N_{G}(x_{s_i}, x_{k+1})\cap A= A\setminus A_i$ and hence
$|N_{G}(x_{s_i}, x_{k+1})\setminus A|\geqslant t-1 -(k-|A_i|)\geqslant |A_i|-1$. We now
remove $ |A_i|-1 $ arbitrary   edges    between     $ x_{k+1} $      and   $ N_{G}(x_{s_i}, x_{k+1})\setminus A $ and join $ x_{s_i} $ to all vertices in  $A_i\setminus\{ x_{s_i}\} $ by  black edges,
resulting   in  $ A\setminus \{x_{s_i}\} \subseteq N_{G}(x_{s_i}, x_{k+1}) $. Therefore,   regardless of whether  $ A_i $ is a clique or  an independent set,
$|N_{G}(x_{s_i}, x_{k+1})\setminus A|\geqslant t-1 -(|A|-1)= t-k $.
Remove $ t-k $ arbitrary edges      between    $ x_{k+1} $      and   $ N_{G}(x_{s_i}, x_{k+1})\setminus A $ and
connect $ x_{k+1} $ to $ t-k $ arbitrary distinct vertices in $ B_i $ by  black edges. Thus, (P.8) and (P.9) hold  yet if  $x_{k+1}$ is added to $A_i$.

Now, update $ A_i $ to $ A_i\cup \{x_{k+1}\} $ and $ A $ to $ A\cup \{x_{k+1}\} $ with the partition $ \{A_1, \ldots, A_{m}\} $.
\renewcommand{\qedsymbol}{$\da$}
\end{proof}

\begin{case}\label{Case III}
$a\in A_{i} $ and $ b\in A_{j} $ for some $i,j$ with $1\leqslant  i< j \leqslant m $.
\end{case}

\begin{proof}[Description]
In view of  (P.1) and  without loss of generality,  we may assume that  $ a=x_{s_i} $ and $ b=x_{s_j} $.
Fix a vertex $ x_r\in A_j$. We know from (P.7), (P.9) and (P.10) that   $ |B_i|>\alpha_r $ and  there are exactly  $ \alpha_r $ black edges  between     $ x_r $      and  $ B_j $. If $ A_{j} $ is an independent set of size $ 3 $  and  $ x_r $ is the third element of $ A_j $, then remove $ \alpha_r-1 $ black edges    between    $ x_r $      and   $ B_j $ and connect $ x_r $ to $ \alpha_r-1 $
arbitrary  vertices in $ B_i $ by black edges. Otherwise, remove $ \alpha_r $ black edges    between    $ x_r $      and   $ B_j $ and connect $ x_r $ to $ \alpha_r $
arbitrary  vertices in $ B_i $ by black edges. Further, in both cases, connect $x_r$ to other vertices in $B_i$ by red edges.

Remove all black edges between $ x_{s_j} $ and $ x_r $  if  $ r>s_j $. The number of such edges is $ k-s_j-q $, where $ q=|A_j\setminus N_{G}[x_{s_j}]| $. Note that,
$|N_{G}(x_{s_i}, x_{s_j})\cap A|=k-p-q-2 $, where $ p=|A_i\setminus N_{G}[x_{s_i}]| $.
We have $ |N_{G}(x_{s_i}, x_{s_j})\setminus A |\geqslant t-1-(k-p-q-2)$.
Remove $ t-1-(k-p-q-2) $ black edges between $ x_{s_j} $ and $ N_{G}(x_{s_i}, x_{s_j})\setminus A $. Since
$(k-s_j-q) + (t-1-(k-p-q-2)) =p+(t-s_j)+1 $, we may  connect $ x_{s_i} $ to all vertices in $(A_i\setminus N_{G}[x_{s_i}])\cup  B_j$ and $ x_{s_j} $ to an arbitrary vertex in $ B_i $ by black edges.
Thus, (P.3) and (P.8)--(P.10) hold yet if $A_i$ is replaced by $A_i\cup A_j$ and $A_j$ is deleted from $\{A_1, \ldots, A_m\}$.

Now, update $ A_i $ to $ A_i\cup A_j $ and update $A_{\ell}, B_{\ell}$ to $A_{\ell-1}, B_{\ell-1}$, respectively, for $\ell=j+1, \ldots, m$. Then,  consider the partition $ \{A_1, \ldots,  A_{m-1}\} $ for $ A $ and the partition $ \{B_1,  \ldots,  B_{m-1}\} $ for $ B $.
\renewcommand{\qedsymbol}{$\da$}
\end{proof}

At the end of each of   Cases \ref{Case I}, \ref{Case II}, and \ref{Case III},  we do the following. In order to establish (P.1) and (P.2), for each vertex  $ w$  with   $ N_{G}(w)\cap A_i\neq \varnothing  $,  join  $ w $ to all vertices  in $ A_i\setminus N_{G}(w) $ by   red edges.
In  order to establish (P.5),    for each vertex  $ w\in C   $  with   $ N_{G}(w)\cap A_i\neq \varnothing  $,  since  there exists at least   a vertex  $ x_r\in A_i $ so  that  the color of  the  edge $wx_r$ is  black,    we may   switch the color of the   edge $wx_r$     with the color of the edge  $ wx_{s_i} $.
Note that, after doing all these   changes,  the number of black edges does not change  and the resulting graph have Properties (P.1)--(P.10). Moreover, by Observation \ref{substit}, $ G $ is still weakly $ K_{2,t} $-saturated.

Now, assume that the  process is terminated. Denote by $ G_{\mathrm{b}}$ the spanning subgraph of $ G$ induced on black edges.
So, in order to  establish  the assertion, we should   show that
$e(G_{\mathrm{b}})\geqslant n-2+\binom{t}{2}$.
For each  $i\geqslant 0$, let $C_i$ be    the set of vertices in $ C $ with   the    distance $i$  from $A$  in $  G_{\mathrm{b}}$ and
let  $\{C_1, \ldots, C_d\}$ be  a partition of $ C$.  Notice that  $C_0=A$. For  any  integer   $i\in\{1,  \ldots, d\}$ and any    vertex $ c\in C_i$,  consider    an    arbitrary edge $ e_c\in   E_{G_{\mathrm{b}}}(\{c\},  C_{i-1})$ and set $ E=\{e_c \, | \,  c\in C \}$.
Denote by  $  G_{\mathrm{c}}$  the spanning  subgraph of $  G_{\mathrm{b}} $ with $ E( G_{\mathrm{c}})=\{x_{s_i}x_r\in E( G_{\mathrm{b}})\,|\, s_i<r\}\cup  E_{ G_{\mathrm{b}}}(A, B) \cup E$.
Finally, set $F= E_{ G_{\mathrm{b}}}(C, V(G))\setminus E_{ G_{\mathrm{c}}}(C, V(G))$.
Note that, in   $  G_{\mathrm{c}}$, every $ C_i$ is an independent set and moreover,  for $ i=1,\ldots, d$,  every vertex in $ C_i$ has exactly one neighbor in $ C_{i-1}$.

Let $ \beta $ and $ \gamma $ be  respectively  the number of independent sets among $ A_1, \ldots, A_m $ of sizes $ 2 $ and $ 3 $. Also, let $\delta $ be $ 1 $ if $ k=t-1 $ and $ 0 $  otherwise.
We have
\begin{align}
e( G_{\mathrm{c}})&= e_{ G_{\mathrm{c}}}(A)+e_{ G_{\mathrm{c}}}(A,B)+e_{ G_{\mathrm{c}}}\big(C, V(G)\big)\nonumber\\[2mm]
&= \left(\sum_{r=1}^{k}\big|\{i \, | \,  s_i<r\}\big| -\beta-2\gamma\right) +\left(\sum_{r=1}^{k}\alpha_r\right) +\left(n-|A|-\sum_{i=1}^{m}|B_i|\right)\nonumber\\[2mm]
&=\left(\sum_{i=1}^{m}\big|\{r \,  |  \, r>s_i\}\big|-\beta-2\gamma\right)+ \left(\sum_{r=1}^{t}(t-r+1)-m +\gamma-\delta\right) + \left(n-k-\sum_{i=1}^{m}(t-s_i)\right)\nonumber\\[2mm]
&=\left(\sum_{i=1}^{m}(k-s_i)-\beta-2\gamma\right)+ \left(\binom{t+1}{2}-m +\gamma-\delta\right) + \left(n-k-\sum_{i=1}^{m}(t-s_i)\right)\nonumber\\[2mm]
&=\bigg(n-2+\binom{t}{2}\bigg)+\bigg((m-1)(k-t-1)-\beta-\gamma-\delta+1\bigg)\label{e(G)}.
\end{align}
We consider the termination states (T.1)--(T.4) in Cases \ref{Case T1h}, \ref{Case T2h}, \ref{Case T3h}, and  \ref{Case T4h}.
In each of these cases,    we will use   \eqref{e(G)} to  establish  that  $e( G_{\mathrm{b}})\geqslant n-2+\binom{t}{2}$.

\begin{case}\label{Case T1h}
(T.1) has happened.
\end{case}

If the second term  in
\eqref{e(G)} is nonnegative, then  there is nothing to prove. So, we may  assume   that $ (m-1)(k-t-1)-\beta-\gamma-\delta+1< 0$.  From  $ k\geqslant t-1+2(m-1)$,     $ \beta +\gamma \leqslant m$,  and    the definition  of $\delta$, we deduce that one of situations
\begin{equation}\label{Ezaf}
\begin{array}{ccc}
\left\{\begin{array}{lll}  k=t-1 \\ m=\beta+\gamma=1  \\  \delta=1      \end{array}\right. & \hspace{1cm} \text{ or  } \hspace{1cm} &
\left\{\begin{array}{lll} k=t+1  \\   m=\beta + \gamma =2 \\  \delta=0        \end{array}\right.
\end{array}
\end{equation}
holds.
Thus, the second term  in
\eqref{e(G)} is equal to $-1$, yielding that     $ e( G_{\mathrm{c}})\geqslant n-3+\binom{t}{2}$.
If $ F\neq\varnothing $, then $e( G_{\mathrm{b}})\geqslant e( G_{\mathrm{c}})+|F|\geqslant n-2+\binom{t}{2} $, we are done. Suppose  by way of contradiction that   $ F=\varnothing $.
Since (T.1) has happened, it follows from \eqref{Ezaf} that
there is an independent set  of size $ t-1 $  among $ A_1, \ldots, A_m $,  concluding  that $ t=3$ or $4 $. We show that $ G[A\cup B]$  is  a bipartite graph.
To see this, note that $A_1, \ldots, A_m$ are independent by \eqref{Ezaf} and so  one may consider the vertex bipartition $\{A_1, B_1\}$ for  $ m=1 $ and the vertex bipartition $\{A_1\cup B_2,  A_2 \cup B_1\}$    for $m=2$.
Using  the  connectivity of $ G $ and  starting by $ G[A\cup B] $, we may add vertices in $ C $ to $G[A\cup B]$  in some order such that in each step  the resulting graph is   bipartite. This means that $ G $ is bipartite  which  is  impossible, since a bipartite  graph is clearly not weakly $ K_{2,t}$-saturated.

\begin{case}\label{Case T2h}
(T.2) has happened.
\end{case}

As $ k=t $, it follows from  \eqref{e(G)} that
$e( G_{\mathrm{c}})= n-2+\binom{t}{2}-(m+\beta+\gamma-2)$. Since $e( G_{\mathrm{b}})\geqslant e( G_{\mathrm{c}})+|F|$, in order to prove the assertion, it suffices to show  that  $|F|\geqslant m+\beta+\gamma-2$.

We may   assume that   either $A\subseteq N_G(u, v)$ or $N_G(u, v)\cap A=A\setminus A_i$   for some  $ i \in \{1, \ldots, m\} $ with $ |A_i|=2$.
To see this, suppose   that  $ |N_G(u, v)\cap  A|\leqslant t-3$ and suppose that
$u$ or    $v$ is not adjacent to $ x_{s_j} $ for some $ j\in \{1, \ldots, m\} $. We find that
$|N_{G}(u, v)\setminus A|\geqslant t-1-|N_{G}(u, v)\cap A| \geqslant  2$.
So, we may   remove two arbitrary edges between   $ v $  and   $ N_{G}(u, v)\setminus A $ and join both  $ u$ and $ v $ to $ x_{s_j} $ by black edges and to all vertices in $ A_j\setminus \{x_{s_j}\}$ by red edges.
By repeating this, we get   $ |N_{G}(u, v)\cap  A|\geqslant t-2$, as desired.

If  $A\subseteq N_{G}(u, v)$, then there are $2m$ black  edges between  $\{u, v\}$ and  $\{x_{s_1}, \ldots,  x_{s_m}\}$ by (P.5)  which only two of them belong to $ E$ and thus $|F|\geqslant 2m-2\geqslant m+\beta+\gamma-2$, we are done.

So, assume   that $N_{G}(u, v)\cap A=A\setminus A_i$   for some  $ i \in \{1, \ldots, m\}  $ with $ |A_i|=2 $.
We have $|N_{G}(u, v)\setminus A|=|N_{G}(u,v)|-|N_{G}(u, v)\cap A|\geqslant t-1-|A\setminus A_i|=1$.
Fix  $w\in N_{G}(u, v)\setminus A$. Obviously, $ w\in C_1\cup C_2$. Then, there are $2m$ black  edges between  $\{u, v\}$ and  $(\{x_{s_1}, \ldots,  x_{s_m}\}\setminus\{x_{s_i}\})\cup\{w\}$ which at most three  of them belong to $ E$.  Hence,  $ |F| \geqslant 2m-3\geqslant m+\beta+\gamma-3$.

Towards   a contradiction, suppose   that   the inequality   $|F|\geqslant m+\beta+\gamma-2$ does not hold. We have   $  m+\beta+\gamma-3\geqslant |F| \geqslant 2m-3\geqslant m+\beta+\gamma-3$  which shows  that   $|F|=2m-3$ and $\beta+\gamma=m$.
It follows from  $ \beta+\gamma=m  $     that   $A_1, \ldots, A_m$ are independent sets.
Since there are $2m-2$ black  edges between  $\{u, v\}$ and  $\{x_{s_1}, \ldots,  x_{s_m}\}\setminus\{x_{s_i}\}$ and  only  two   of them belong to $ E$, we deduce that $ |F\cap E_{ G_{\mathrm{b}}}(A, C)|\geqslant 2m-4 $.
Also, at least one of the black   edges $uw$ or  $vw$ belong to $F$, meaning  that  $ |F\cap E_{ G_{\mathrm{b}}}(C)|\geqslant 1$.
Now,  from   $|F|=2m-3$, $ |F\cap E_{ G_{\mathrm{b}}}(A, C)|\geqslant 2m-4 $, and  $ |F\cap E_{ G_{\mathrm{b}}}(C)|\geqslant 1$, we  derive     that  $uv\notin E(G)$ and $ |F\cap E_{ G_{\mathrm{b}}}(C)|= 1$. Thus,
$w\in C_2$ and $ G_{\mathrm{c}}[C]= G_{\mathrm{b}}[C]-e$, where $e\in\{uw, vw\}$.

Denote by  $G'$  be the graph obtained from $G$ by joining $u$ to all vertices in $N_G(v)\setminus N_G(u)$ and joining  $v$ to all vertices in $N_G(u)\setminus N_G(v)$. Set  $ A'=A\cup \{u,v\}$, $ C'=C\setminus \{u,v\}$,  $ A_j'=A_j$ for $ j=1,\ldots, m$,   and     $  A_{m+1}'=\{u,v\} $.
We know that  $A_1', \ldots, A_{m+1}'$ are independent sets.
For each  $j\geqslant 0$, let $C_j'$ be    the set of vertices in $ C' $ with    the   distance $j$  from $A'$  in $  G'$ and
let  $\{C_1', \ldots, C_{d'}'\}$ be  a partition of $ C'$.
Since  $ G'[C']= G_{\mathrm{c}}[C']$, we observe    in     $  G'$ that  $ C_1', \ldots, C_{d'}'$   are independent sets    and moreover,  for $ j=2,\ldots, d'$,  every vertex in $ C_j'$ has exactly one neighbor in $ C_{j-1}'$.
Further, for any  vertex   $c\in C_1' $, there is an index  $j\in\{1, \ldots, m+1\}$ such that  $ N_{G'}(c)=A_j'$.
Using these features,    the following statements  are  easily obtained   for  two    arbitrary distinct  vertices $y, z \in V(G')$.
\begin{itemize}
\item[(i)] Let $y\in C'$ and $ z\in A'\cup C'$. Then,  $N_{G'}(y, z)$ is  one of $\varnothing$, $\{c\}$, or $A_j'$ for some  vertex   $ c \in C'$ and  integer  $j\in \{1,\ldots, m+1\}$.
Hence, $|N_{G'}(y, z)|\leqslant  t-2$.
\item[(ii)] Let $y,z\in A'$. If $y\in A_{j}'$ and $ z\in A_{\ell}'$ for some indices  $ j\neq \ell$, then $N_{G'}(y, z)\subseteq A'\setminus (A_{j}'\cup A_{\ell}')$. Thus, $|N_{G'}(y, z)|\leqslant t-2$.
\end{itemize}
As $G'$ is not a complete graph and  there is no pair $\{ y, z\} $   of vertices of $G'$
such that  $ |N_{G'}(y, z)|\geqslant t-1 $ and $ N_{G'}(y)\setminus\{z\}\neq N_{G'}(z)\setminus\{y\}$, one derives  that $ G'$ and therefore $G$  are  not weakly $ K_{2,t}$-saturated,   a contradiction.

\begin{case}\label{Case T3h}
(T.3) has happened.
\end{case}

As $ k=t+1 $, it follows from  \eqref{e(G)} that
$e( G_{\mathrm{c}})=  n-2+\binom{t}{2}-(\beta+\gamma-1)$. Since $e( G_{\mathrm{b}})\geqslant e( G_{\mathrm{c}})+|F|$, in order to prove the assertion, it is sufficient  to show  that   $|F|\geqslant \beta+\gamma-1$.

We may  assume that    either $A\subseteq N_{G}(u, v)$ or $N_{G}(u, v)\cap A=A\setminus A_i$   for some  $ i \in \{1, \ldots, m\}  $ with $ |A_i|\in\{2, 3\}$.
To see this, suppose   that  $ |N_{G}(u, v)\cap  A|\leqslant t-3$ and
$u$ or    $v$ is not adjacent to $ x_{s_j} $ for some $ j\in \{1, \ldots, m\} $. We find that
$|N_{G}(u, v)\setminus A|\geqslant t-1-|N_{G}(u, v)\cap A| \geqslant  2$.
So, we may   remove two arbitrary edges between   $ v $  and   $ N_{G}(u, v)\setminus A $ and join both  $ u$ and $ v $ to $ x_{s_j} $ by black edges and to all vertices in $ A_j\setminus \{x_{s_j}\}$ by red edges.
By repeating this, we get   $ |N_{G}(u, v)\cap  A|\geqslant t-2$, as desired.

If  $A\subseteq N_{G}(u, v)$, then there are $2m$ black  edges between  $\{u, v\}$ and  $\{x_{s_1}, \ldots,  x_{s_m}\}$ by (P.5)  which only two of them belong to $ E$ and thus  $|F|\geqslant 2m-2\geqslant \beta+\gamma-1$, we are done.

So, assume   that $N_{G}(u, v)\cap A=A\setminus A_i$   for some  $ i \in \{1, \ldots, m\}  $ with $ |A_i|\in\{2, 3\}$.
Thus, there are $2m-2$ black  edges between  $\{u, v\}$ and  $\{x_{s_1}, \ldots,  x_{s_m}\}\setminus\{x_{s_i}\}$ which only two   of them belong to $ E$.
Therefore,  $ |F| \geqslant 2m-4\geqslant m-2\geqslant \beta+\gamma-2$.

Working toward  a contradiction,   suppose  that   the inequality    $|F|\geqslant \beta+\gamma-1$ is  not valid. We  have
$\beta+\gamma-2\geqslant |F| \geqslant 2m-4\geqslant m-2\geqslant \beta+\gamma-2$
which means that
$m=\beta+\gamma=2$ and $F=\varnothing$.
It follows from  $ m=\beta+\gamma=2 $     that    $A_1, A_2$ are independent sets and $A=A_1\cup A_2$.
Since $ t+1=|A_1|+|A_2|\leqslant 2\min \{3, t-2\}$, one  concludes  that $ t=5$ and therefore   $ |A_1|=|A_2|=3$.
Furthermore, it follows from $F=\varnothing$ that  $ |N_{G}(u,v)\cap C|=0 $ and so   $ |N_{G}(u,v)|=|N_{G}(u,v)\cap A| =3$,  contradicts with $ |N_G(u,v)|\geqslant t-1$.

\begin{case}\label{Case T4h}
(T.4) has happened.
\end{case}

As $ k=t+1 $, it follows from  \eqref{e(G)} that
$e( G_{\mathrm{c}})=  n-2+\binom{t}{2}-(\beta+\gamma-1)$.  Since $e( G_{\mathrm{b}})\geqslant e( G_{\mathrm{c}})+|F|$, in order to prove the assertion, it is enough   to show  that $|F|\geqslant \beta+\gamma-1$.
Before proceeding,  we point out that    $2m\leqslant |A_1|+\cdots+|A_m|\leqslant m(t-2)$ and so   $2m\leqslant t+1\leqslant m(t-2)$  which forces that  $t\geqslant5$.

From   $u\in A$ and  in view of (P.1), we may assume that  $ u=x_{s_i}$ for some $i \in \{1, \ldots, m\} $.
Suppose that  $ v $ is not adjacent to $ x_{s_j} $ for some $ j \in  \{1, \ldots,  m\}\setminus\{i\} $. We have
$|N_{G}(u, v)\setminus A|\geqslant t-1-|N_{G}(u, v)\cap A|
\geqslant t-1 -(k-|\{u\}\cup A_j|)= |A_j|-1\geqslant 1$.
So, we may   remove an  arbitrary  edge  between  $ v $   and   $ N_{G}(u, v)\setminus A $ and join $v$ to $ x_{s_j} $ by a black edge and to  all vertices in  $ A_j\setminus \{x_{s_j}\} $ by red edges.
By repeating this, we  get
$A\setminus A_i \subseteq N_{G}(u, v)$.
Accordingly,   there are $m-1$ black  edges between  $v$ and  $\{x_{s_1}, \ldots,  x_{s_m}\}\setminus\{x_{s_i}\}$ which only one  of them belongs to $E$.
Therefore,  $ |F| \geqslant m-2\geqslant \beta+\gamma-2$.

Towards a contradiction, suppose that the inequality  $|F|\geqslant \beta+\gamma-1$ does not hold.
We have   $\beta+\gamma-2\geqslant|F| \geqslant m-2\geqslant \beta+\gamma-2$ which means that  $|F|=m-2$ and $\beta+\gamma=m$.
It follows from  $ m=\beta+\gamma $     that   $A_1, \ldots, A_m$ are independent sets.
Also,  it follows from   $|F|=m-2$  that $uv\notin E(G)$ and $ G_{\mathrm{b}}[C]= G_{\mathrm{c}}[C]$. The latter equality shows that $N_{G}(u,v)\cap C=\varnothing$. Since $ A_i$ is an independent set, we get $N_{G}(u,v)=A\setminus A_i$  which in turn yields that $|A_i|=2$.

Denote by  $G'$    the graph obtained from $G$ by joining both  vertices in  $A_i$ to all vertices in $N_G(v)\setminus N_G(u)$ and joining  $v$ to all vertices in $N_G(u)\setminus N_G(v)$. Set  $ A'=A\cup \{v\}$,  $ C'=C\setminus \{v\}$, $ A_i'= A_i\cup\{v\}$,  and $ A_j'=A_j$ for any $ j\in\{1, \ldots, m\}\setminus  \{i\}$.
We know that  $A_1', \ldots, A_{m+1}'$ are independent sets.
For each  $j\geqslant 0$,  let    $C_j'$       be   the set of vertices in $ C' $ with    the   distance $j$  from $A'$  in $  G'$ and
let  $\{C_1', \ldots, C_{d'}'\}$ be  a partition of $ C'$.
As   $ G'[C']= G_{\mathrm{c}}[C']$, we observe    in      $  G'$ that  every $ C_j'$ is an independent set and moreover,  for $ j=2,\ldots, d'$,  every vertex in $ C_j'$ has exactly one neighbor in $ C_{j-1}'$.
Further, for any  vertex   $c\in C_1' $, there is an index $j\in\{1, \ldots, m\}$ such that  $ N_{G'}(c)=A_j'$.
Using  these features and noting that   $t\geqslant5$,    the following statements  are  straightforwardly  obtained   for  two    arbitrary distinct  vertices $y, z \in V(G')$.
\begin{itemize}
\item[(i)] Let $y\in C'$ and $ z\in A'\cup C'$. Then,  $N_{G'}(y, z)$ is  one of $\varnothing$, $\{c\}$, or $A_j'$ for some  vertex  $ c \in C'$ and   integer   $j\in \{1,\ldots, m\}$.
Thus, $|N_{G'}(y, z)|\leqslant t-2$.
\item[(ii)] Let $y,z\in A'$. If $y\in A_{j}'$ and $ z\in A_{\ell}'$ for some indices  $ j\neq \ell$, then $N_{G'}(y, z)\subseteq A'\setminus (A_{j}'\cup A_{\ell}')$. Hence, $|N_{G'}(y, z)|\leqslant t-2$.
\end{itemize}
As  $G'$ is not a complete graph and  there   is no pair $\{ y, z\} $   of vertices of $G'$
such that  $ |N_{G'}(y, z)|\geqslant t-1 $ and $ N_{G'}(y)\setminus\{z\}\neq N_{G'}(z)\setminus\{y\}$, one  deduces   that $ G'$ and therefore $G$  are  not weakly $ K_{2,t}$-saturated,   a contradiction.

The proof  is  completed here.
\end{proof}

To complete the proof of Theorem \ref{main.thm},  it remains  to establish    that    $\mathrm{wsat}(n, K_{2, t})\geqslant n-1+\binom{t}{2}$ for every
integers $n, t$ when  $t$ is even   and    $t+2 \leqslant n \leqslant 2t-2 $.
The following two Lemmas   have  a crucial role in the    proof of  the   last  lemma.

\begin{lemma}\label{XYZ}
Let $t\geqslant 2$ and $G$ be a graph with the  vertex set   $V(G)=X\cup Y \cup Z$ for three sets $X, Y, Z$ with $|X|\geqslant2$ and $|Y|\geqslant1$ and the edge set $E(G)=\{uv \, | \,  (u,v)\in (X,X)\cup (X,Y)\cup (Y,Z)\}$.
If $G$ is  weakly $ K_{2,t} $-saturated, then either $|X|\geqslant t$ or $|Y|\geqslant t-1$ and $|X|+|Z|\geqslant t$.
\end{lemma}

\begin{proof}
Assume that $G$  is a weakly $ K_{2,t} $-saturated graph and $|X|\leqslant t-1$. At the beginning of the weakly $K_{2,t}$-saturation process, there are two vertices $a, b\in V(G)$ such that $|N_G(a,b)|\geqslant t-1$ and $N_G(a)\setminus \{b\}\neq N_G(b)\setminus \{a\}$. According to the structure of $G$ and noting that $N_G(a)\setminus \{b\}\neq N_G(b)\setminus \{a\}$, we deduce that    $\{a, b\}$ is not contained in one of $X, Y, Z$.  If $a\in X$ and $b\in Y$, then $N_G(a, b)= X\setminus \{a\}$ and if $a\in Y$ and $b\in Z$, then $N_G(a,b)=\varnothing$. The both cases are impossible,  as $|N_G(a,b)|\geqslant t-1$. So, we may assume that $a\in X$ and $b\in Z$.
Since $N_G(a, b)\subseteq Y$, one concludes that $|Y|\geqslant t-1$.
Now, as  $X\setminus\{a\} \subseteq N_G(a)$, all edges in $E_G(b, X\setminus\{a\})$  can be added to $G$ and therefore, as  $|X|\geqslant 2$, all edges in  $E_G(X, Z)$ can be added to $G$. Also,  it follows from  $Z \subseteq N_G(a)$ that  all edges in $E_G(b, Z)$ and hence all edges in $E_G(Z)$  can be added to $G$.

Denote the resulting graph by $ G' $. Note that  $V(G')=Y\cup Y'$ and $E(G')=\{uv \, | \,  (u, v)\in (Y, Y')\cup (Y', Y')\}$, where $Y'=X\cup Z$.
In  order to proceed the  weakly $ K_{2,t} $-saturation process,    two vertices $c, d\in V(G')$ must exist  such that $|N_{G'}(c,d)|\geqslant t-1$ and $N_{G'}(c)\setminus \{d\}\neq N_{G'}(d)\setminus \{c\}$. According to the structure of $G'$ and noting that $N_{G'}(c)\setminus \{d\}\neq N_{G'}(d)\setminus \{c\}$, we may assume that $c\in Y$ and $d\in Y'$.
As $N_{G'}(c, d)\subseteq Y'\setminus \{d\}$, one deduces  that $|Y'\setminus \{d\}|\geqslant t-1$. This means that $|X|+|Z|\geqslant t$, completing the proof.
\end{proof}

It is straightforward to check that the converse of Lemma \ref{XYZ} is also valid.

\begin{lemma}\label{degone}
Let $t\geqslant3$ and let $ G$ be a weakly $K_{2,t}$-saturated graph of order  $ n $  with  $  n\leqslant 2t-2 $. Assume that   $ v $ is   a degree one vertex in $ G $. Then,  $ G-v$ is also  weakly $ K_{2,t} $-saturated.
\end{lemma}

\begin{proof}
Do a weakly $ K_{2,t} $-saturation  process on $ V(G)\setminus \{v\} $ by joining nonadjacent vertices as far  as possible. It follows that  $N_G(x)\setminus\{v, y\}=N_G(y)\setminus\{v, x\}$ for every two vertices $x,y \in V(G)\setminus \{v\}$ with $ |N_G(x,y) |\geqslant t-1$.
Suppose by way of contradiction
that $ G-v  $ is not a complete graph. Define the relation $ \approx $ on $  V(G)\setminus \{v\} $ as $  x\approx y $ if $ N_G(x)\setminus\{v, y\}=N_G(y)\setminus\{v, x\}$. Clearly, $ \approx $ is an equivalence
relation on $  V(G)\setminus \{v\} $. Note that  the equivalence classes are cliques  or independent sets and the  connections  between two distinct equivalence classes are all  present or all  absent. Moreover, in view of       $n\leqslant 2t-2$ and $ \deg_G(v)=1$, we observe that  any  independent equivalence class is of size  at most $t-2$. Further,
any  clique  equivalence class is of size  at most $t-1$. Otherwise,   we observe that   $K_{t+1}$ is a subgraph of $ G-v  $ and, by the connectivity of $G$ and applying   Lemma \ref{pendent}, we deduce that $G-v$ is a    weakly $ K_{2,t} $-saturated graph,  a contradiction.

Assume that $\EuScript{A} $ is the equivalence class containing the neighbor of $ v $.
Note that $\EuScript{A}\not= V(G)\setminus \{v\} $.
We continue  the weakly $ K_{2,t} $-saturation process on $ G $ by joining $ v $ to all vertices in $ \EuScript{A} $.
As  $ G $ is not complete, there is a vertex $w$ such that  $ |N_{G}(v,w)|\geqslant t-1 $ and $ N_{G}(v)\setminus\{w\}\neq N_{G}(w)\setminus\{v\}$.
Since $ N_{G}(v)= \EuScript{A} $, we
deduce that   $\EuScript{A}$ is a clique of size $t-1$ and moreover,   $ \EuScript{A}\subseteq N_{G}(w)$ implies  that   $w\notin \EuScript{A}$.
Let $ \EuScript{B}$   be   the equivalence class containing $w$.
We claim that $ \EuScript{B} $ is an independent set. By contradiction,  suppose  that $w$ has a neighbor   $b\in\EuScript{B}$.
Letting    $a\in \EuScript{A}$, we have   $(\EuScript{A}\setminus \{a\})\cup\{w\}\subseteq N_{G}(a,b) $.
This implies that   $ a\approx b $  which   contradicts $\EuScript{A}\neq\EuScript{B}$, proving the claim.    Now, add $v$ to  $\EuScript{B}$, join $v$ to all vertices in   $N_{G}(w)\setminus N_{G}(v) $, and  join $w$ to all vertices in   $N_{G}(v)\setminus N_{G}(w) $.
The resulting graph is not still complete, since   $\EuScript{B} $ is an   independent set.
Thus, to proceed with the weakly $ K_{2,t} $-saturation  process, there should  be two vertices $ x \napprox y $   such that $ |N_G(x,y)|\geqslant t-1 $ and
$v\in N_{G}(x,y) $.

Since    $ x \napprox y $, at most one of $x, y$ belongs  to $\EuScript{A}$.
First, suppose  that $ x,y \notin \EuScript{A}$.
From  $ x \napprox y $  and  $ |\EuScript{A}|=t-1 $, we conclude that $ \EuScript{A} \nsubseteq N_G(x,y)$ which means that $\EuScript{A} \cap  N_{G}(x,y)=\varnothing$. Then, it follows
from $x, y\notin \EuScript{A}$ and  $  \{x,y\}\cup \EuScript{A} \cup N_{G}(x,y)\subseteq V(G) $ that
$n\geqslant 2t$, a contradiction.  Next, suppose  without loss of generality that $ x\in \EuScript{A}$ and $ y\notin \EuScript{A}$. As
$v\in N_G(x,y)$, we deduce  that $ w \in N_G(x,y)$ and so $(\EuScript{A}\setminus \{x\}) \cup \{w\} \subseteq N_G(x,y)$. This shows that $ |N_G(x,y)|\geqslant t-1$ and therefore $ x\approx y$, a contradiction.
This contradiction completes the proof.
\end{proof}

The following lemma establishes  a    lower     bound on    $\mathrm{wsat}(n, K_{2, t})$ for even     $t$.

\begin{lemma}\label{event}
Let $t\geqslant4$ be even and let $n$ be an integer with  $t+2 \leqslant n \leqslant 2t-2 $. Then,  $\mathrm{wsat}(n, K_{2, t})\geqslant n-1+\binom{t}{2}$.
\end{lemma}

\begin{proof}
For $t=4$, the assertion  follows from Theorem \ref{Kst}. So,     assume that $t\geqslant6$ is fixed. Working toward   a contradiction,
consider    a weakly $ K_{2,t} $-saturated graph $ G_0 $ which is a counterexample with the minimum possible  order.
In view of Lemma \ref{generallower},  we find that  $ e(G_0)=n_0-2+\binom{t}{2}$, where $ n_0=|V(G_0)|$.
We have  $t+3\leqslant n_0\leqslant 2t-2$ by Theorem \ref{Kst} and
$\delta(G_0)\geqslant2$     by Lemma  \ref{degone}.
Using Lemma \ref{pendent}, if we attach a new pendent vertex to a vertex of $ G_0 $, then the resulting graph is also weakly $ K_{2,t}$-saturated
whose number of vertices  is one more  than the number of vertices of $G_0$  and whose number of edges is one more  than the number of edges of $G_0$.
Attach $ t^3$ new pendent vertices to each vertex of $ G_0 $ and call
the resulting graph by  $ G$. By assuming $ n=|V(G)|$, we have
$e(G)= n-2+\binom{t}{2}$.

We  apply the process defined in  the proof of Lemma \ref{generallower} to $G$ and assume that the process has terminated. As the number of black edges does not change during the process, $ e( G_{\mathrm{b}})=n-2+\binom{t}{2}$.
We need a careful exploration  of  termination states (T.1)--(T.4) to get  a contradiction. We will do it below   by distinguishing  Cases \ref{evenT1}, \ref{evenT2}, \ref{evenT3}, and \ref{evenT4}.

\begin{case}\label{evenT1}
(T.1) has happened.
\end{case}

Since  $ e( G_{\mathrm{b}})=n-2+\binom{t}{2}$, it follows from   \eqref{e(G)} that
\begin{equation}\label{nega}
(m-1)(k-t-1)-\beta-\gamma-\delta+1\leqslant0.
\end{equation}
Using    $ k\geqslant t-1+2(m-1)$, $ \beta +\gamma \leqslant m$, and $  \delta\leqslant  1$,  we derive  from \eqref{nega} that $ m=1$ or $2$. First,   assume that  $ m=1$. Then,   we deduce  from  $ k\geqslant t-1\geqslant 5$ that  $ \beta=\gamma =0$. Thus,   it follows  from   \eqref{nega}   that      $ \delta=1$ and so  $ k=t-1$. Next, assume that  $ m=2$.  As  $ k\geqslant t-1+2(m-1)$, we have   $ k\geqslant t+1$ and so $ \delta =0$. Hence,  it follows  from   \eqref{nega}   that  $\beta + \gamma \geqslant 1$. If $ \beta+\gamma=2$, then we get $ k\leqslant 6$ which contradicts $k\geqslant  t+1\geqslant7$.  Therefore, $ \beta+\gamma=1$ and  so    we find from  \eqref{nega}  that $ k=t+1 $. This   forces  that   $ \beta=1$ and $ \gamma  =0$.

The  discussion done above indicates  that  the second term  in   \eqref{e(G)} is equal to $0$, implying   that $ G_{\mathrm{b}}= G_{\mathrm{c}} $. We know that there is no red edge in $ C$, meaning that $ G[C]=G_{\mathrm{c}}[C]$. From this and since  $ G_0$ is a subgraph of $ G$ with $ \delta(G_0)\geqslant 2$, we deduce  that  $ V(G_0)\cap C\subseteq C_1$. Further, as   $ G_0$ is a subgraph of $ G$  and  $ G_{\mathrm{b}}= G_{\mathrm{c}} $, we yield   for any  vertex $ c\in V(G_0)\cap C$ that  there is an index $i\in\{1, \ldots, m\}$ such  that $ N_{G_0}(c)\cap A\subseteq N_{G}(c)\cap A = A_i$.

We are  now  ready to describe  the structure of $G_0$.
If  $m=1$, then    $ G_0$ is obviously  a subgraph of   the graph  $\mathbbmsl{G}=K_{s}\vee \overline{K_{n_0-s}}$ for some $ s\leqslant t-1$. But, this is  a contradiction, since  $\mathbbmsl{G}$ and therefore $ G_0$ are not weakly  $ K_{2,t} $-saturated.
So,  suppose that $ m=2 $. As (T.1)  has happened and   $k=t+1$, we may assume without  loss of generality that
$|A_1|=2$ and $|A_2|=t-1$.
It follows from $ \beta =1$   that $ A_1$ is   an  independent set.
Now, it is easily seen  that  $G_0$ is a subgraph of the graph $\mathbbmsl{H}$, depicted in Figure \ref{figH}, with  $X\subseteq A_2$, $ Y_1 \subseteq A_1$, and $ |V(\mathbbmsl{H})|=n_0$. Since $ |X|\leqslant t-1$ and $ n_0\leqslant 2t-2$,  Lemma \ref{XYZ} implies that the described graph  $\mathbbmsl{H} $ and therefore $ G_0$ are not weakly $ K_{2,t}$-saturated, a contradiction.

\begin{case}\label{evenT2}
(T.2) has happened.
\end{case}

Since  $k=t$, it follows from  \eqref{e(G)} that
$n-2+\binom{t}{2}=e( G_{\mathrm{b}})\geqslant e( G_{\mathrm{c}})+|F|\geqslant n-2+\binom{t}{2}-(m+\beta +\gamma-2)+|F|$ and therefore  $ |F|\leqslant m+\beta+\gamma-2\leqslant 2m-2$.
As we proved in Case  \ref{Case T2h},     either $A\subseteq N_{G}(u, v)$ or $N_{G}(u, v)\cap A=A\setminus A_i$   for some  $ i\in\{1, \ldots, m\} $ with $ |A_i|=2$.

Assume that  $A\subseteq N_{G}(u, v)$.  So,    there are $2m$ black  edges between  $\{u, v\}$ and  $\{x_{s_1}, \ldots,  x_{s_m}\}$ by (P.5)  which only two of them belong to $ E$. This  yields
that      $ |F|\geqslant 2m-2 $ and  thus     $ |F|=2m-2$ and      $\beta+\gamma=m$.

Assume that     $N_{G}(u, v)\cap A=A\setminus A_i$   for some  $ i\in\{1, \ldots, m\} $ with $ |A_i|=2$. We have
$|N_{G}(u, v)\setminus A|=|N_{G}(u,v)|-|N_{G}(u, v)\cap A|\geqslant t-1-|A\setminus A_i|=1$.
Consider an arbitrary vertex  $w\in N_{G}(u, v)\setminus A$. As at most one of the edges $ uw$ and $ vw$ belongs to $ E$, we may assume  without loss of generality  that  $uw\in F$.  Clearly, $ w\in C_1\cup C_2$. There are $2m$ black  edges between  $\{u, v\}$ and  $(\{x_{s_1}, \ldots,  x_{s_m}\}\setminus\{x_{s_i}\})\cup\{w\}$ which at most three  of them belong to $ E$.
Therefore,     $ |F|$ is equal to  either  $2m-3$ or $2m-2$. Since $ |F|\leqslant m+\beta+\gamma-2\leqslant 2m-2$, we find that $ m-1\leqslant\beta+\gamma\leqslant m $.

Let $R=F\setminus E_G(\{u,v\},\{x_{s_1}, \ldots,  x_{s_m}\})$. According to  what we  saw   above,   $|R|\leqslant 2$.
Set  $ A'=A\cup \{u,v\}$,  $ C'=C\setminus \{u,v\}$,  $ A_j'=A_j$ for  $ j=1,\ldots, m$,  and $ A_{m+1}'=\{u, v\} $.
We distinguish the following  six cases.
\begin{itemize}[label=*****, wide=0pt, leftmargin=*]
\item[(I.1)] $A\subseteq N_G(u,v)$, $ |F|=2m-2$, and $R=\varnothing$. In this case,  $ A_1', \ldots, A_{m+1}'$ are independent sets.
\item[(I.2)] $N_G(u, v)\cap A=A\setminus A_i$, $ A_i $ is an independent set,  $ |F|=2m-3$, and $R=\{uw\}$. In this case,    $ A_1',  \ldots,  A_{m+1}'$ are independent
sets except possibly  for  $A_j'$, where   $j\in\{1,  \ldots, m\}\setminus\{i\}$.
\item[(I.3)] $N_G(u, v)\cap A=A\setminus A_i$, $ A_i$ is a clique, $ |F|=2m-3$, and $R=\{uw\}$. In this case, $ A_i'$ is  a clique and $ A_1', \ldots, A_{i-1}', A_{i+1}', \ldots,  A_{m+1}'$ are independent sets.
\item[(I.4)] $N_G(u, v)\cap A=A\setminus A_i$, $ |F|=2m-2$, and $R=\{uv, uw\}$. In this case, $ A_1', \ldots, A_m'$ are independent sets and $ A_{m+1}'$ is a clique.
\item[(I.5)] $N_G(u, v)\cap A=A\setminus A_i$, $ |F|=2m-2$, and $R=\{ab,uw\}$ for some $a\in A'$ and $b\in C'$. In this case, $ A_1', \ldots, A_{m+1}'$ are independent sets.
\item[(I.6)] $N_G(u, v)\cap A=A\setminus A_i$, $ |F|=2m-2$, and $R=\{b_1b_2,uw\}$ for some $b_1,b_2\in  C'$. In this case, $ A_1', \ldots, A_{m+1}'$ are independent sets.
\end{itemize}

We define a supergraph $G'$ of $ G$ as follows.
Denote by  $G'$  the graph obtained from $G$ by joining $u$ to all vertices in $N_G(v)\setminus N_G[u]$ and joining  $v$ to all vertices in $N_G(u)\setminus N_G[v]$.
For any  $j\geqslant 0$, let  $C_j'$     be    the set of vertices in $ C' $ with    the   distance $j$  from $A'$  in $  G'$ and
let  $\{C_1', \ldots, C_{d'}'\}$ be  a partition of $ C'$.
In   the cases (I.1)--(I.5), we have
$E(G'[C'])= E( G_{\mathrm{c}}[C']) $   and   therefore  we observe  in       $ G'$ that  $ C_1', \ldots, C_{d'}'$   are independent sets   and moreover,    every vertex in $ C_j'$ has exactly one neighbor in $ C_{j-1}'$   for    $ j=2,\ldots, d'$.
In the case (I.6), we have
$E(G'[C'])= E( G_{\mathrm{c}}[C'])\cup \{b_1b_2\} $    and therefore  we observe  in      $ G'$ that      $ C_1', \ldots, C_{d'}'$, all  except  probably for one,   are independent sets   and moreover,    every vertex in $ C_j'$ has exactly one neighbor in $ C_{j-1}'$  for all     $ j\in\{2,\ldots, d'\}$  except  probably for   one.
Further,  for every vertex $c\in C_1'\setminus \{b\} $, there is an  index  $j\in\{1, \ldots, m+1\}$ such that  $ N_{G'}(c)\cap A'=A_j'$.

First,  we  consider   the cases (I.1), (I.3), (I.4), and (I.6).
In view of   the  structure of $G'$ described above,  the following statements  are  easily   obtained   for  two    arbitrary distinct  vertices $y, z \in V(G')$.
\begin{itemize}
\item[(i)] Let $y\in C'$ and $ z\in A'\cup C'$. Then,  $N_{G'}(y, z)$ is
one of $\varnothing$, $\{c\}$, $\{c_1, c_2\} $, $A_j'$,  or $A_j'\cup\{c\}$  for some vertices $ c, c_1, c_2 \in C'$ and index $j\in \{1,\ldots, m+1\}$. Hence,   $|N_{G'}(y, z)|\leqslant  4$.
\item[(ii)] Let $y,z\in A'$. Assume that $y\in A_{j}'$ and $ z\in A_{\ell}'$ for
some $ j\neq \ell$. In the cases (I.1) and (I.6), we have 	$N_{G'}(y, z)\subseteq A'\setminus (A_{j}'\cup A_{\ell}')$. In the case (I.3), we have $N_{G'}(y, z)\subseteq A'\setminus (A_{j}'\cup A_{m+1}')$ if $ \ell=i$ and
$N_{G'}(y, z)\subseteq A'\setminus (A_{j}'\cup A_{\ell}')$ if $ i\notin\{j, \ell\}$.
In the case (I.4), we have $N_{G'}(y, z)\subseteq A'\setminus (A_{j}'\cup A_{i}')$ if $ \ell=m+1$ and
$N_{G'}(y, z)\subseteq A'\setminus (A_{j}'\cup A_{\ell}')$ if $ m+1\notin\{j, \ell\}$.
Thus, in any   case, $|N_{G'}(y, z)|\leqslant t-2$.
\end{itemize}
As $G'$ is not a complete graph and  there is no pair
$\{ y, z\} $  of vertices of $G'$  such that  $ |N_{G'}(y, z)|\geqslant t-1 $ and $ N_{G'}(y)\setminus\{z\}\neq N_{G'}(z)\setminus\{y\}$, one  finds  that $ G'$ and therefore $G$  are  not weakly $ K_{2,t}$-saturated,  a contradiction.

Next,  we    consider the cases (I.2) and (I.5).
Since   $ G_0$ is   a subgraph of $G'$ with  $ \delta(G_0)\geqslant 2$, we conclude that  $ V(G_0)\cap C'\subseteq C_1'$.
In  the case  (I.2), set   $\widetilde{A}=A'$  and in    the case (I.5),  set   $\widetilde{A}=A'\cup\{b\}$.
Now, it is  easily  seen  that  $G_0$ is a spanning subgraph of the graph $\mathbbmsl{H}$,  depicted in Figure \ref{figH},  with $ X\subseteq\widetilde{A}\setminus (A_i'\cup A_{m+1}')$ and $ Y_1\subseteq A_i'\cup A_{m+1}'$.
As $ |X|\leqslant t-1$ and  $ n_0\leqslant 2t-2$,  Lemma \ref{XYZ} implies that the described graph  $\mathbbmsl{H} $ and therefore $ G_0$ are not weakly $ K_{2,t}$-saturated, a contradiction.

\begin{case}\label{evenT3}
(T.3) has happened.
\end{case}

Since  $k=t+1$, it follows from  \eqref{e(G)} that
$n-2+\binom{t}{2}=e( G_{\mathrm{b}})\geqslant e( G_{\mathrm{c}})+|F|\geqslant n-2+\binom{t}{2}-(\beta +\gamma-1)+|F|$. Therefore, $ |F|\leqslant \beta+\gamma-1\leqslant m-1$. As we proved in  Case  \ref{Case T3h},
either $A\subseteq N_{G}(u, v)$ or $N_{G}(u, v)\cap A=A\setminus A_i$   for some  $ i\in\{1, \ldots, m\} $ with $ |A_i|\in\{2, 3\}$.

Suppose  that $A\subseteq N_{G}(u, v)$.  Then,  there are $2m$ black  edges between  $\{u, v\}$ and  $\{x_{s_1}, \ldots,  x_{s_m}\}$ by (P.5)  which only two of them belong to $ E$. This
yields    that   $ |F|\geqslant 2m-2 $  which along with    $ |F|\leqslant  m-1$    gives   $m\leqslant1$, a contradiction.

So, we may assume that   $N_{G}(u, v)\cap A=A\setminus A_i$   for some  $ i $ with $ |A_i|\in\{2, 3\}$.
Then,  there are $2m-2$ black  edges between  $\{u, v\}$ and  $\{x_{s_1}, \ldots,  x_{s_m}\}\setminus\{x_{s_i}\}$ which only two   of them belong to $ E$.
This gives      $ |F|\geqslant 2m-4 $ which along with $ |F|\leqslant \beta+\gamma-1\leqslant m-1$  leads  to    $2m-3\leqslant\beta+\gamma\leqslant m$.  Hence, either $1\leqslant\beta+\gamma\leqslant m =2$ or $m=\beta+\gamma=3$.

If $m=\beta+\gamma=2 $, then $ t+1=|A_1|+|A_2|\leqslant 6$ which contradicts $ t\geqslant 6$.
So, assume that either $m=2$ and $\beta+\gamma=1$ or $m=\beta+\gamma=3$. In both cases, $ |F|=2m-4$ and so
$ F$ is contained in the set of the black  edges between  $\{u, v\}$ and  $\{x_{s_1}, \ldots,  x_{s_m}\}\setminus\{x_{s_i}\}$. Therefore, $G[C]= G_{\mathrm{b}}[C]= G_{\mathrm{c}}[C] $. From this and since  $ G_0$ is  a subgraph of $ G$ with  $ \delta(G_0)\geqslant 2$,  we conclude that  $ V(G_0)\cap C\subseteq C_1$.
As each  edge in $ F $ is incident  to either $ u$ or $ v$, we observe in $G_{\mathrm{b}}$ that  every vertex in  $ C\setminus \{u,v\}$ has exactly one neighbor in $ A$. Therefore, for every vertex $ c\in V(G_0)\cap C\setminus\{u,v\}$, we have $ N_{G_0}(c)\cap A\subseteq N_{G}(c)\cap A = A_j$ for some $j\in\{1, \ldots, m\}$.
It results in that  $G_0$ is a
subgraph of the graph $\mathbbmsl{H}$,  depicted in Figure \ref{figH}, with $ X\subseteq A\setminus A_i$, $ Y_1\subseteq A_i\cup \{u,v\}$,  and $ |V(\mathbbmsl{H})|=n_0$.
Since $ |X|\leqslant t-1$ and $ n_0\leqslant 2t-2$,  Lemma \ref{XYZ} implies that the described graph  $\mathbbmsl{H} $ and therefore $ G_0$ are not weakly $ K_{2,t}$-saturated, a contradiction.

\begin{case}\label{evenT4}
(T.4) has happened.
\end{case}

Since  $k=t+1$, it follows from  \eqref{e(G)} that
$n-2+\binom{t}{2}=e( G_{\mathrm{b}})\geqslant e( G_{\mathrm{c}})+|F|\geqslant n-2+\binom{t}{2}-(\beta +\gamma-1)+|F|$. Therefore, $ |F|\leqslant \beta+\gamma-1\leqslant m-1$.
As we proved in Case   \ref{Case T4h},   $A\setminus A_i \subseteq N_{G}(u, v)$ by assuming    $ u=x_{s_i}$. So,  there are $m-1$ black  edges between  $v$ and  $\{x_{s_1}, \ldots,  x_{s_m}\}\setminus\{x_{s_i}\}$ which only one  of them belongs to $ E$. Hence,   $|F|\geqslant m-2$ which derives that  $  |F|$ equals either  $ m-2$ or $ m-1$.

We show that $ N_G(u,v)\cap A_i=\varnothing$. Suppose otherwise. This forces that  $ A_i$ is a clique and $ uv\in E(G)$. The latter implies that $ |F|=m-1$ and so $ \beta +\gamma=m$. In particular, $ A_i$ is an independent set, a contradiction.

If $N_G(u, v)=A\setminus A_i $, then $ t-1\leqslant |N_G(u,v)|=t+1-|A_i|$ and so $ |A_i|=2$.

If   $N_{G}(u, v)\neq A\setminus A_i$, then there is a vertex $w\in N_G(u, v)\cap C_1$. Since $ v,w\in C_1$, we have $vw\in F$. As $ F$ also  contains $ m-2$ edge between $v$ and  $\{x_{s_1}, \ldots,  x_{s_m}\}\setminus\{x_{s_i}\}$, we conclude that $ |F|=m-1$ and so $ \beta+\gamma=m$. Note that  $ |F|=m-1$ forces that $N_{G}(u, v)= (A\setminus A_i)\cup \{w\} $.

Let $R=F\setminus E_G(\{u,v\},\{x_{s_1}, \ldots,  x_{s_m}\})$. According to  what we  saw   above,    $|R|\leqslant 1$.
Set  $ A'=A\cup \{v\}$,  $ C'=C\setminus \{v\}$,  $ A_i'=A_i\cup\{v\}$,  and $ A_j'=A_j$ for any $ j\in\{1, \ldots, m\}\setminus\{ i\}$.
We distinguish  the following five  cases.
\begin{itemize}[label=*****, wide=0pt, leftmargin=*]
\item[(J.1)] $N_G(u, v)=A\setminus A_i $, $ |F|=m-2$, and $R=\varnothing$.  In this case, $ |A_i|=2$ and all of $ A_1', \ldots, A_m'$ except possibly for one  are independent sets.
\item[(J.2)] $N_G(u, v)=A\setminus A_i $, $ |F|=m-1$, and $R=\{uv\}$. In this case, $ A_i $ is an independent set  of size $2$ and   $  A_1', \ldots, A_{i-1}',   A_{i+1}', \ldots, A_m'$ are independent sets.
\item[(J.3)] $N_G(u, v)=A\setminus A_i $, $ |F|=m-1$, and $R=\{ab\}$ for some $a\in A'$ and $b\in C'$. In this case, $|A_i|=2 $ and $A_1', \ldots, A_m'$ are independent sets.
We divide this case to the following three subcases.
\begin{itemize}[label=*******, wide=0pt, leftmargin=*]
\item[(J.3.1)] There is $ A_{\ell}'$ such that $ |A_{\ell}'|=3$ and $ N_G(b)\cap A_{\ell}'=\varnothing$.
\item[(J.3.2)] Among  $ A_1', \ldots, A_m'$, there are exactly two sets $ A_{j_1}'$ and $ A_{j_2}'$  which  meet  $N_{G}(b)$. In addition, $ t\geqslant 8$ and  $ |A_{j_1}'|=|A_{j_2}'|=3$.
\item[(J.3.3)] $t=6$, $ m=3$, $ |A_1'|=|A_2'|=3$, and $ |A_3'|=2$. The vertex $ b$ has neighbors in both $ A_1', A_2'$ and no neighbor in $ A_3'$.
\end{itemize}
\item[(J.4)] $N_G(u, v)=A\setminus A_i $,  $ |F|=m-1 $, and $R=\{b_1b_2\}$ for some  $b_1,b_2\in C'$. In this case,  $ A_1', \ldots, A_m'$ are independent sets.
\item[(J.5)] $N_{G}(u, v)= (A\setminus A_i)\cup \{w\} $,  $ |F|=m-1 $, and $R=\{vw\}$. In this case,  $ A_1', \ldots, A_m'$ are independent sets.
\end{itemize}

We define a supergraph $G'$ of $ G$ as follows.
Denote by   $G'$   the graph obtained from $G$ by joining  $ c$ to all vertices in $N_G(v)\setminus N_G[c]$ and joining  $v$ to all vertices in $N_G(c)\setminus N_G[v] $ for every vertex $ c\in A_i$.
Note that, if either $ uv\in E(G)$ or $ A_i $ is a clique in $ G $, then $ A_i'$ is a clique in $ G'$.
For any  $j\geqslant 0$, let  $C_j'$     be    the set of vertices in $ C' $ with    the   distance $j$  from $A'$  in $  G'$ and
let  $\{C_1', \ldots, C_{d'}'\}$ be  a partition of $ C'$.
In   the cases (J.1)--(J.3) and (J.5), we have
$E(G'[C'])= E( G_{\mathrm{c}}[C']) $   and   therefore  we observe  in       $ G'$ that  $ C_1', \ldots, C_{d'}'$   are independent sets   and moreover,    every vertex in $ C_j'$ has exactly one neighbor in $ C_{j-1}'$   for    $ j=2,\ldots, d'$.
In the case (J.4), we have
$E(G'[C'])= E( G_{\mathrm{c}}[C'])\cup \{b_1b_2\} $    and therefore  we observe  in      $ G'$ that      $ C_1', \ldots, C_{d'}'$,  all  except  probably for one,     are independent sets   and moreover,    every vertex in $ C_j'$ has exactly one neighbor in $ C_{j-1}'$  for all   $ j\in\{2,\ldots, d'\}$ except  probably for   one.
Further, for every vertex $c\in C_1'\setminus \{b\} $, there is an index  $j\in\{1, \ldots, m\}$ such that  $ N_{G'}(c)\cap A'=A_j'$.

First, we consider  the cases (J.1), (J.2), and (J.3.1).
We claim  that   there exists an independent set $ A_{\ell}'$ of size $ 3 $.
There is nothing to prove in the case  (J.3.1).
In  the cases (J.1) and  (J.2), if $ A_i'$ is an  independent set, then we let   $ \ell=i$. Otherwise, since $ t-1=|A_1'|+ \cdots +|A_{i-1}'|+ |A_{i+1}'|+ \cdots +|A_m'|$ is odd and  $  A_1', \ldots, A_{i-1}', A_{i+1}', \ldots, A_m'$ are independent sets of sizes  $2 $ or $ 3$, we find  an index    $ \ell\in\{1, \ldots, m\}\setminus\{ i\}$ such that $ A_{\ell}'$ is an independent set of size $ 3 $, as we claimed.
Since  $ G_0$ is a subgraph of $ G'$ with  $ \delta(G_0)\geqslant 2$,    we should have   $ V(G_0)\cap C'\subseteq C_1'$.
Now, it is straightforwardly   seen  that  $G_0$ is a spanning subgraph of the graph $\mathbbmsl{H}$,  depicted in Figure \ref{figH},  with $ X\subseteq A'\setminus A_{\ell}'$ and $ Y_1\subseteq A_{\ell}'$.
As $|X|\leqslant t-1$ and  $ n_0\leqslant 2t-2$,  Lemma \ref{XYZ} implies that the described graph  $\mathbbmsl{H} $ and therefore $ G_0$ are not weakly $ K_{2,t}$-saturated, a contradiction.

Next, we consider  the cases (J.3.2), (J.4),  and (J.5).
In view of   the  structure of $G'$ described above,  the following statements  are  easily   obtained   for  two    arbitrary distinct  vertices $y, z \in V(G')$.
\begin{itemize}
\item[(i)] Let $y\in C'$ and $ z\in A'\cup C'$.
Then,  $N_{G'}(y, z)$ is  one of $\varnothing$, $\{c\}$, $\{c_1, c_2\} $, $A_{j}'$, $A_j'\cup\{c\}$, $ A_{j_1}' \cup A_{j_2}'$, where $ c, c_1, c_2 \in C'$, $j\in \{1,\ldots, m\}$, and $ j_1, j_2$ are given in the case (J.3.2). This shows that  $|N_{G'}(y, z)|\leqslant  \max\{4, |A_{j_1}'\cup A_{j_2}'| \}$.
As $ t\geqslant 8 $ in the case (J.3.2)  and $ t\geqslant 6$ in the cases (J.4) and (J.5), one  deduces   that  $|N_{G'}(y, z)|\leqslant  t-2$.
\item[(ii)] Let $y,z\in A'$. Assume that $y\in A_{\ell_1}'$ and $ z\in A_{\ell_2}'$ for some $ \ell_1\neq \ell_2$. Then, $N_{G'}(y, z)$ is
either $A'\setminus (A_{\ell_1}'\cup A_{\ell_2}')  $ or $ (A'\cup \{b\})\setminus (A_{j_1}'\cup A_{j_2}') $. Note that the latter one occurs in the case   (J.3.2) whenever  $ \ell_1=j_1$ and $ \ell_2=j_2$. So, in any case, $|N_{G'}(y, z)|\leqslant t-2$.
\end{itemize}
As $G'$ is not a complete graph and  there is no pair $\{ y, z\}$    of vertices of $G'$
such that  $ |N_{G'}(y, z)|\geqslant t-1 $ and $ N_{G'}(y)\setminus\{z\}\neq N_{G'}(z)\setminus\{y\}$, we deduce  that $ G'$ and therefore $G$ are  not weakly $ K_{2,t}$-saturated, a contradiction.

Finally, we consider the case  (J.3.3). Let     $ p\in A_3'$. It follows from $A_1'\cup A_2'\subseteq N_{G'}(b, p)$ that   $ |N_{G'}(b, p)|\geqslant t-1$.
We define a supergraph $ G''$ of $ G'$ as follows.
Denote by  $G''$    the graph obtained from $G'$ by joining both  vertices in $A_3'$  to all vertices in $N_G(b)\setminus N_G(p)$ and joining  $b$ to all vertices in $N_G(p)\setminus N_G(b) $.
Set  $ A''=A'\cup \{b\}$,  $ C''=C'\setminus \{b\}$,  $ A_1''=A_1'$,  $ A_2''=A_2'$, and  $ A_3''=A_3'\cup \{b\}$.
For any  $j\geqslant 0$, let  $C_j''$     be    the set of vertices in $ C'' $ with    the   distance $j$  from $A''$  in $  G''$ and
let  $\{C_1'', \ldots, C_{d''}''\}$ be  a partition of $ C''$.
As
$ E(G''[C''])= E( G_{\mathrm{c}}[C'']) $,   we observe  in       $ G''$ that  $ C_1'', \ldots, C_{d''}''$   are independent sets   and moreover,    every vertex in $ C_j''$ has exactly one neighbor in $ C_{j-1}''$   for    $ j=2,\ldots, d''$.
Further, for every vertex $c\in C_1'' $, there is an index  $j\in\{1, 2, 3\}$ such that $ N_{G''}(c)\cap A''=A_j''$.
Using  these features,    the following statements  are  easily   obtained   for  two    arbitrary distinct  vertices $y, z \in V(G'')$.
\begin{itemize}
\item[(i)] Let $y\in C''$ and $ z\in A''\cup C''$.
Then,  $N_{G''}(y, z)$ is  one of $\varnothing$, $\{c\}$, or $A_{j}''$ for some vertex  $ c \in C''$ and index  $j\in \{1, 2, 3\}$.
Hence,  $|N_{G''}(y, z)|\leqslant  3$.
\item[(ii)] Let $y,z\in A''$. Assume without loss of generality  that $y\in A_{1}''$ and $ z\in A_{2}''$. So, $N_{G''}(y, z)= A_3''$ and thus  $|N_{G''}(y, z)|= 3$.
\end{itemize}
As  $t=6$,  there exists  no pair $\{ y, z\}$    of vertices of $G''$
such that  $ |N_{G''}(y, z)|\geqslant t-1 $ and $ N_{G''}(y)\setminus\{z\}\neq N_{G''}(z)\setminus\{y\}$. But,   $G''$ is not  a  complete graph, so    $ G''$ and therefore $G$ are  not weakly $ K_{2,t}$-saturated, a contradiction.

The proof   is   completed here.
\end{proof}

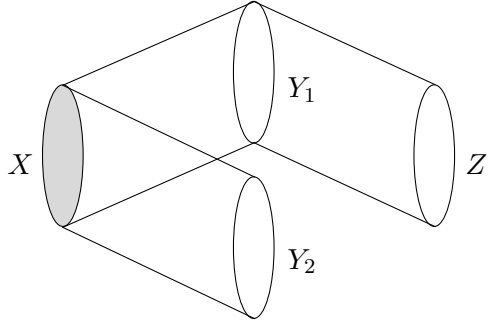
\begin{figure}[H]
\begin{center}
\begin{tikzpicture}[line cap=round,line join=round,>=triangle 45,x=1.0cm,y=1.0cm,scale=0.6]
\draw [rotate around={90:(-2.85,-8.56)},line width=0.4pt,color=black,fill=gray!30] (-2.85,-8.56) ellipse (1.57cm and 0.45cm);
\draw [rotate around={90:(1.38,-10.6)},line width=0.4pt,color=black] (1.38,-10.6) ellipse (1.57cm and 0.45cm);
\draw [rotate around={90:(1.38,-6.71)},line width=0.4pt,color=black] (1.38,-6.71) ellipse (1.57cm and 0.45cm);
\draw [rotate around={90:(5.38,-8.66)},line width=0.4pt,color=black] (5.48,-8.66) ellipse (1.57cm and 0.45cm);
\draw (-2.85,-7)-- (1.4,-5.14);
\draw (-2.85,-10.15)-- (1.4,-8.27);
\draw (-2.85,-7)-- (1.4,-9.035);
\draw (-2.85,-10.15)-- (1.4,-12.175);
\draw (1.38,-5.14)-- (5.4,-6.987);
\draw (1.38,-8.27)-- (5.4,-10.135);
\draw (-4.3,-8.3) node[anchor=north west] {$X$};
\draw (1.9,-6.6) node[anchor=north west] {$Y_1$};
\draw (5.85,-8.3) node[anchor=north west] {$Z$};
\draw (1.9,-10.4) node[anchor=north west] {$Y_2$};
\end{tikzpicture}
\caption{The graph  $\mathbbmsl{H}$. The set  $ X$ is a clique and the sets  $ Y_1, Y_2, Z$ are independent.
Every vertex in $ X$ is adjacent to every vertex in $ Y_1\cup Y_2$ and every vertex in $ Z$ is adjacent  to every vertex in $ Y_1$.}\label{figH}
\end{center}
\end{figure}

We end the paper  here by pointing out that
Theorem \ref{main.thm}  is concluded from     Lemmas  \ref{lowerr}, \ref{generallower},  and \ref{event}.


\begin{thebibliography}{99}


\bibitem{Alon}    N. Alon, An extremal problem for sets with applications to graph theory,   J. Combin. Theory Ser. A      40  (1985)  82--89.

\bibitem{Balogh} J. Balogh, B. Bollob\'as,  R. Morris, Graph bootstrap percolation, Random Structures Algorithms  41 (2011) 413--440.

\bibitem{Bella}    B. Bollob\'{a}s,   Weakly  $k$-saturated graphs, Beitr\"age zur Graphentheorie   (Kolloquium, Manebach, 1967), Teubner, Leipzig, 1968, pp. 25--31.

\bibitem{Borowiecki} M. Borowiecki,  E. Sidorowicz, Weakly $\mathcal{P}$-saturated graphs, Discuss. Math. Graph Theory 22 (2002) 17--30.

\bibitem{cui} Y. Cui, L.  Pu,   Weak saturation numbers of  $K_{2, t}$  and $K_p\bigcup K_q$,  AKCE Int. J. Graphs Comb. 16 (2019)  237--240.

\bibitem{Faud.3}  B.L. Currie,  J.R. Faudree, R.J. Faudree, J.R. Schmitt, A survey of minimum saturated graphs, Electron. J. Combin.     (2021) \#DS19.

\bibitem{Faudree} R.J. Faudree, R.J. Gould, M.S. Jacobson, Weak saturation numbers for sparse graphs, Discuss. Math. Graph Theory 33 (2013) 677--693.

\bibitem{Frankl}   P. Frankl, An extremal problem for two families of sets,   European J. Combin.  3  (1982)  125--127.

\bibitem{Kalai2}    G. Kalai, Hyperconnectivity of graphs,    Graphs Combin.     1  (1985)  65--79.

\bibitem{Kalai1}    G. Kalai, Weakly saturated graphs are rigid,  in: Convexity and graph theory (Jerusalem, 1981), North-Holland Math. Stud., 87, Ann. Discrete Math., 20, North-Holland, Amsterdam, 1984, pp.  189--190.

\bibitem{Kalinichenko} O. Kalinichenko, M. Zhukovskii, Weak saturation stability, European J. Combin. 114 (2023) 103777.

\bibitem{Kronenberg} G. Kronenberg, T.  Martins, N.  Morrison,   Weak saturation numbers of complete bipartite graphs in the clique,  J. Combin. Theory Ser. A 178 (2021)  105357.

\bibitem{Lovasz}   L. Lov\'{a}sz, Flats in matroids and geometric graphs, in:  Combinatorial surveys (Proc. Sixth British Combinatorial Conf., Royal Holloway Coll., Egham, 1977),  Academic Press, London, 1977, pp. 45--86.

\bibitem{Yu}    J.  Yu, An extremal problem for sets: a new approach via Bezoutians,   J. Combin. Theory Ser. A     62  (1993)  170--175.


\end{thebibliography}
\end{document}